\date{}
\documentclass[twoside,a4paper,12pt,leqno]{article}
\usepackage{amsmath,amssymb,amsopn,amsthm,color,mathrsfs,array,hhline,epsfig,multirow}
\usepackage[e]{esvect}
\makeatletter
\renewcommand\@pnumwidth{1.55em}
\renewcommand\@tocrmarg{2.55em}
\renewcommand\@dotsep{4.5}
\renewcommand\tableofcontents{%
    \section*{Contents}\ \vskip 5pt%
    \@starttoc{toc}%
    \addtocontents{toc}{\begingroup\protect\small}%
    \AtEndDocument{\addtocontents{toc}{\endgroup}}%
    }
\renewcommand*\l@part[2]{%
  \ifnum \c@tocdepth >-2\relax
    \addpenalty\@secpenalty
    \addvspace{2.25em \@plus\p@}%
    \begingroup
      \setlength\@tempdima{3em}%
      \parindent \z@ \rightskip \@pnumwidth
      \parfillskip -\@pnumwidth
      {\leavevmode
       \large \bfseries #1\hfil \hb@xt@\@pnumwidth{\hss #2}}\par
       \nobreak
       \if@compatibility
         \global\@nobreaktrue
         \everypar{\global\@nobreakfalse\everypar{}}%
      \fi
    \endgroup
  \fi}
\renewcommand*\l@section{\@dottedtocline{1}{0pt}{1.5em}}
\renewcommand*\l@subsection{\@dottedtocline{2}{1.5em}{2.3em}}
\renewcommand*\l@subsubsection{\@dottedtocline{3}{3.8em}{3.2em}}
\renewcommand*\l@paragraph{\@dottedtocline{4}{7.0em}{4.1em}}
\renewcommand*\l@subparagraph{\@dottedtocline{5}{10em}{5em}}
\makeatother
\newcommand{\vertbar}{\smash{\vrule height 15pt depth 6pt width .6pt}}
\renewcommand{\vec}[1]{\ensuremath{\mathchoice
                     {\mbox{\boldmath$\displaystyle#1$}}
                     {\mbox{\boldmath$\textstyle#1$}}
                     {\mbox{\boldmath$\scriptstyle#1$}}
                     {\mbox{\boldmath$\scriptscriptstyle#1$}}}}%
\makeatletter
  \@addtoreset{equation}{section}
\makeatother
\numberwithin{equation}{section}
\newtheorem{Def}[equation]{Definition}

\newtheorem{Thm}[equation]{Theorem}
\newtheorem{Cor}[equation]{Corollary}
\newtheorem{Prop}[equation]{Proposition}
\newtheorem{Lem}[equation]{Lemma}
\newtheorem{Rem}[equation]{Remark}
\newtheorem{Exer}[equation]{Exercise}
\newtheorem{Example}[equation]{Example}
\newenvironment{definition}{\begin{Def}\ \rm}{\end{Def}}
\newenvironment{theorem}{\begin{Thm}\ }{\end{Thm}}
\newenvironment{proposition}{\begin{Prop}\ }{\end{Prop}}

\newenvironment{lemma}{\begin{Lem}\ }{\end{Lem}}

\newenvironment{example}{\begin{Example}\ \rm}{\end{Example}}
\newenvironment{remark}{\begin{Rem}\ \rm}{\end{Rem}}

\newenvironment{prop-def}{\begin{Prop-Def}\ }{\end{Prop-Def}}

\newenvironment{working-hypothesis}{\begin{WH}\ \rm}{\end{WH}}

\def\bvec#1{\ensuremath{\mathchoice
                     {\mbox{\boldmath\(\displaystyle#1\)}}
                     {\mbox{\boldmath\(\textstyle#1\)}}
                     {\mbox{\boldmath\(\scriptstyle#1\)}}
                     {\mbox{\boldmath\(\scriptscriptstyle#1\)}}}}%
\DeclareMathOperator*{\Res}{\mathrm{Res}}
\renewcommand{\i}{\bvec{i}}
\renewcommand{\P}{\mathrm{P}}
\newcommand{\I}{\mathrm{I}}

\newcommand{\underbarl}[1]{\lower 1.4pt \hbox{\underbar{\raise 1.4pt \hbox{#1}}}}
 
\newcommand{\tp}[1]{{}^t\hskip 0pt#1}

\newcommand{\llg}{\langle\hskip -1.5pt\langle}
\newcommand{\rrg}{\rangle\hskip -1.5pt\rangle}

\newcommand{\jrac}[2]{\dfrac{\lower 2pt\hbox{\(#1\)}}{\raise 2pt\hbox{\(#2\)}}}

\newcommand{\tj}[2]{\ensuremath{\mathchoice
                     {{{t_{#2}}^{\!\raise .20em\hbox{\mbox{\scalebox{.4}{\((#1)\)}}}}}} 
                     {{{t_{#2}}^{\!\raise .20em\hbox{\mbox{\scalebox{.4}{\((#1)\)}}}}}} 
                     {{{t_{#2}}^{\raise .18em\hbox{\mbox{\scalebox{.3}{\((#1)\)}}}}}} 
                     {{{t_{#2}}^{\raise .15em\hbox{\mbox{\scalebox{.2}{\((#1)\)}}}}}} 
                     }}
\newcommand{\ti}[1]{{\tj{#1}{}\,}}

\newcommand{\br}[1]{^{\mbox{\raisebox{3pt}{\scalebox{.5}{\(\left[#1\right]\)}}}}}

\newcommand{\twolower}[1]{\hbox{\raise -2pt \hbox{\(#1\)}}}

\newcommand{\hookdownarrow}{\kern 0pt
\hbox{\vbox{\offinterlineskip \kern 0pt
      \hbox{\(\cap\)}\kern 0pt
      \hbox{\hskip 3.5pt\(\downarrow\)}\kern 0pt}}
}  

\newcommand{\lqq}{\lq\lq}

\newcommand{\struth}{\vrule height 1.1em depth .7em width 0em}
\newcommand{\strutHigh}{\vrule height 13pt depth 0pt width 0pt}

\newcommand{\boxit}[1]{\vbox{\hrule\hbox{\vrule\kern3pt
    \vbox to 43pt{\hsize 182pt\kern3pt#1\eject\kern3pt\vfill}
    \kern1pt\vrule}\hrule}} 
\newcommand{\wt}{\mathrm{wt}}

\newcommand{\adots}{\mathinner{\mkern1mu\raise 0pt\vbox{\kern 7pt\hbox{.}}\mkern2mu
    \raise 4pt\hbox{.}\mkern2mu\raise 7pt\hbox{.}\mkern1mu}}
\newcommand{\wdots}{\vbox{\baselineskip 4pt \lineskiplimit 0pt
    \kern 6pt\hbox{.}\hbox{.}\hbox{}}}
\newcommand{\mdots}{\vbox{\baselineskip 4pt \lineskiplimit 0pt
    \kern 6pt\hbox{}\hbox{.}\hbox{.}}}

\newcommand{\rslt}{\mathrm{rslt}}

\newcounter{item}

\title{Hurwitz Integrality of Power Series Expansion \\ of The Sigma Function for a Plane Curve}
\author{Yoshihiro \^Onishi}

\begin{document}
\allowdisplaybreaks
\maketitle
\begin{abstract}
This paper shows Hurwitz integrality of the coefficients of expansion at the origin of 
the sigma function \(\sigma(u)\) associated to a certain plane curve 
which should be called a plane telescopic curve. 
For the prime \(2\), the expansion of \(\sigma(u)\) is not Hurwitz integral, 
but \(\sigma(u)^2\) is. 
This paper clarifies the precise structure of this phenomenon. 
Throughout the paper, computational examples for the trigonal genus three curve (\((3,4)\)-curve) 
\(y^3+(\mu_1x+\mu_4)y^2+(\mu_2x^2+\mu_5x+\mu_8)y=x^4+\mu_3x^3+\mu_6x^2+\mu_9x+\mu_{12}\) (\(\mu_j\) are constants) 
are given. 
\end{abstract}
\newpage
\noindent
{\large\bf Introduction}
\vskip 10pt
\noindent

{\bf 1}. More than twenty years ago, Buchstaber, Enolskii, and Leykin 
started to investigate certain family of plane algebraic curves, 
which they call \((n,s)\)-curves. 
Around that time, Shinji Miura also wrote some papers for the purpose of coding theory 
on larger class of curves includes many non-plane curves and the family above. 
His papers are published only in Japanese. 
The paper \cite{miura_1998} is one of his most important papers. 
An important subclass of his class is called the class of telescopic curves, 
in which any curve has many good properties in a contrast to the other curves. 
Especially, they have good number theoretic properties as I described below. 
For example, like the coefficient of the elliptic curves, 
any coefficient of a telescopic curve has weight, 
and good generalization of the sigma function of Weierstass can be attached. 
In this paper, we only treat {\it plane} curves introduced by Buchstaber, Enolskii, and Leykin, 
and show that the power series expansion of the attached sigma function is Hurwitz integral 
(see the definition \ref{hurwitz_integral}) outside the prime \(2\). 
Moreover, we show full detail of how the prime \(2\) appears in the expansion. 

Our result includes the case for the Weierstrass sigma function. 
Investigations for this case are given in papers by 
Bannai-Kobayashi \cite{bannai-kobayashi_2006}, Mazur-Stein-Tate \cite{MST_2006}, 
Mazur-Tate \cite{mazur-tate_1991} from other points of view. 
Our result in this paper might be generalized to any telescopic curves. 
\par
{\bf 2}. For any Abelian variety \(A\) with a choice of a theta divisor 
which does not pass through the origin of \(A\), 
some \(p\)-adic theta functions are constructed in \cite{breen} and \cite{barsotti_1970}. 
In such case, the power series expansion of those theta series around the origin 
are Hurwitz integral. 
However, for any telescopic curve, the standard theta divisor passes through the origin of the Jacobian variety. 
Since, in general, the theta divisor has some singularity at the origin, 
the author was not able to prove Hurwitz integrality by a similar method 
to \cite{breen} and \cite{barsotti_1970}. 
In fact, the expansion of the sigma function for a telescopic curve is not completely Hurwitz integral and 
has some defect at the prime \(2\). 
So, the author believe that to publish this paper is worthwhile. 
\par
{\bf 3}. The sigma function for a non-singular algebraic curve of genus \(g\) is 
an entire function of \(g\) variables which is invariant under any change (modular transform) of 
the symplectic base of the homology group of the curve and whose zeroes are of order \(1\) and 
exactly along the pull-back image of the \((g-1)\)st symmetric product of copies of the curves 
with respect to the map given by modulo the lattice. 
Because of its invariance under the modular transform, 
the expansion of the sigma function around the origin is expressed only in terms 
of the coefficients of the defining equation of the curve. 
In fact, the coefficients of the expansion are polynomials over the rationals 
of the coefficients of the defining equation of the curve. 
Actually, the expansion is Hurwitz integral over the ring generated 
by the coefficients of the defining equation over the {\it integers} outside the prime \(2\), 
and is treated, as it is without modifications, not only over the complex numbers 
but also over \(p\)-adic numbers or other rings. 
\par
{\bf 4}. In Section \ref{preliminaries}, we introduce basic situation and notations. 
In Section \ref{the_det_expression}, following Nakayashiki \cite{nakayashiki_2010}, 
we show an expression of the sigma function as a determinant of infinite size (the tau function) 
times an exponential function. 
In Section \ref{the_klein_2-form}, we prove the integrality of an expansion of 
the canonical Klein's fundamental 2-form. 
In Section \ref{the_exp_part}, after analyzing the exponential part, 
we complete the proof. 
As multivariate sigma functions are not well-known, 
we present fairly detailed proofs as well as many examples. 
\par
{\bf 5}. Acknowledgments: The author is grateful for significant advice from A. Nakayashiki. 
Concerning the part {\bf 2} above, the author has a help from Shin-ichi Kobayashi. 
J.C. Eilbeck showed me several examples from his data. 
The author sincerely appreciate this. 
\par
{\bf 6}. 
We use the usual notation \(\mathbb{Z}\), \(\mathbb{Q}\), and \(\mathbb{C}\) for 
the ring of integers, the field of rationals, the field of complex numbers. 
Finally, we note here that the 1-forms of the first kind are included our definition 
of 1-forms of the second kind, in this paper. 
\vskip 50pt
\noindent
\tableofcontents
\newpage
\section{Statement of the main theorem}\label{preliminaries}

Let \(e\) and \(q\) be fixed two positive integers such that \(e<q\) and \(\gcd(e,q)=1\). 
Let us define, for these integers, a polynomial of \(x\) and \(y\) 
  \begin{equation}\label{5.01}
  f(x,y)=y^e+p_1(x)y^{e-1}+\cdots+p_{e-1}(x)y-p_e(x),
  \end{equation}
where  \(p_j(x)\) is a polynomial of \(x\) of degree \(\lceil\tfrac{jq}{e}\rceil\) or below
and its coefficients are denoted as
  \begin{equation}\label{5.02}
  \begin{aligned}
  p_j(x)&=\sum_{k:jq-ek>0}\mu_{jq-ek}\,x^k \ \ \ (1\leqq j\leqq e-1),\\
  p_e(x)&=x^q+\mu_{e(q-1)}x^{q-1}+\cdots+\mu_{eq}. 
  \end{aligned}
  \end{equation}
The base ring over which we set up situation is quite general. 
However, for simplicity we start by letting it an algebraically closed field 
and assuming \(\mu_i\) to be constants belonging to the field. 
Let \(\mathscr{C}\) be the projective curve defined by 
  \begin{equation}\label{plane_miura}
  f(x,y)=0
  \end{equation}
having unique point \(\infty\) at infinity. 
This should be called an {\it \((e,q)\)-curve} following to Buchstaber, Enolskii, and Leykin, or 
{\it plane Miura curve} after the paper \cite{miura_1998}. 
If this is non-singular, the genus of it is  \((e-1)(q-1)/2\). 
Whether \(\mathscr{C}\) is non-singular or singular, we constantly denote this by \(g\): 
\begin{equation*}
g=\frac{(e-1)(q-1)}2.
\end{equation*}
As general elliptic curve is defined by an equation of the form
  \begin{equation}\label{5.035}
  y^2+(\mu_1x+\mu_3)y=x^3+\mu_2x^2+\mu_4x+\mu_6,
  \end{equation}
our curve \(\mathscr{C}\) is a natural generalization of elliptic curves. 

We introduce weight as follows:
\begin{equation*}
\wt(\mu_j)=-j, \ \ \ 
\wt(x)=-e, \ \ \ 
\wt(y)=-q. 
\end{equation*}
Then all the equations for functions, power series, differential forms, and so on in this paper are of homogeneous weight. 
Needless to say that \(\wt(f(x,y))=-eq\). 

Here, we shall explain what is the sigma function roughly. 
Let \(J=\mathrm{Jac}(\mathscr{C})\) be the Jacobian variety of \(\mathscr{C}\). 
Now, we take the {\it standard theta divisor}  \(\Theta\br{g-1}\subset J\) 
that is strictly defined in (\ref{theta_cycle}) below. 
Then the {\it sigma function} attached to  \(\mathscr{C}\) is a meromorphic section of
the sheaf \(\mathscr{O}(\Theta\br{g-1})\). 
In other words, it is an entire function defined over the universal (Abelian) covering \(\mathbb{C}^g\) of \(J\) 
which is invariant under the modular transformation associated to the natural symplectice structure 
and the second derivatives of logarithm of it are periodic function with respect to the lattice 
determined by the chosen symplectic base of \(H_1(\mathscr{C},\mathbb{Z})\). 
So that, the sigma function has zeroes of order \(1\) along the pull-back divisor of \(\Theta\br{g-1}\subset\mathrm{Jac}(\mathscr{C})\) 
under \(\mathbb{C}^g\rightarrow J\) and invariant under the transformation by the natural action of \(\mathrm{Sp}(2g,\mathbb{Z})\). 
Such function is uniquely determined up to overall multiplication of a constant. 
The modular invariance implies that the power series expansion of the sigma function is 
expressed only by data on  \(\mathscr{C}\) such as its coefficients. 
In fact, the power series expansion of the sigma function around the origin 
has polynomial coefficients in \(\mu_j\)'s over \(\mathbb{Q}\) (\cite{nakayashiki_2008}). 

In order to define (see \ref{def_sigma}) the sigma function classically, 
we take the point \(\infty\) to be the base point of \(\mathscr{C}\), 
and fix a symplectic base of \(H_1(\mathscr{C},\mathbb{Z})\) and of \(H^1(\mathscr{C},\mathbb{C})\) as we explain below. 
As usual, the base of later one consist of \(g\) differential forms of the first kind 
and \(g\) ones of the second kinds having a pole only at \(\infty\). 

Assuming that we have defined the sigma function strictly, we shall state the main result after 
the following two more definitions. 

\begin{definition}
For each coefficient \(\mu_j\) of \(f(x,y)\), 
we denote \({\mu_j}'=\frac12\mu_j\) if \(\mu_j\) is the coefficient of monomial 
with odd power of \(x\) times odd power of \(y\), 
and \({\mu_j}'=\mu_j\) otherwise. 
Moreover, we denote by \(\bvec{\mu}'\) the set of all \({\mu_j}'\). 
\end{definition}
\begin{example}
If \((e,q)=(3,4)\), then
\begin{equation*}
f(x,y)=y^3+(\mu_1x+\mu_4)y^2+(\mu_2x^2+\mu_5x+\mu_8)y-(x^4+\mu_3x^3+\mu_6x^2+\mu_9x+\mu_{12}). 
\end{equation*}
Hence, \(\vec{\mu}'=\{\mu_1,\mu_4,\mu_8,\mu_2,\tfrac12\mu_5,\mu_3,\mu_6,\mu_9, \mu_{12}\}\). 
\end{example}
\begin{definition}\label{hurwitz_integral}
Let \(R\) be an integral domain with characteristic \(0\) and \(z_1\), \(z_2\),\(\cdots\), \(z_n\) be 
indeterminates. Then we define
\begin{equation*}
R\llg z_1, z_2, \cdots, z_n \rrg
=\Bigg\{\sum_{j_1=1}^{\infty}\sum_{j_2=1}^{\infty}\cdots\sum_{j_n=1}^{\infty}
a_{j_1j_2\cdots {j_n}}\frac{\,{z_1}^{j_1}{z_2}^{j_2}\cdots{z_n}^{j_n}\,}{j_1!j_2!\cdots j_n!}\ \Bigg|\ 
a_{j_1j_2\cdots{j_n}}\in R\,\Bigg\}. 
\end{equation*}
This is a commutative ring by usual addition and multiplication. 
If a power series with coefficients in the quotient field of \(R\) belongs to this ring, 
it is said to be {\it Hurwitz integral} over \(R\). 
\end{definition}

The sigma function \(\sigma(u)\) is a function of \(g\) variables. 
The entries of the variable \(u\) are suffixed by the sequence of positive integers \(\{w_g,w_{g-1},\cdots,w_1\}\) 
that are not expressed as \(ae+bq\) with positive integers \(a\) and \(b\), 
which are called the {\it Weierstrass gaps} for \((e,q)\), 
and is denoted as \(u=(u_{w_g},u_{w_{g-1}},\cdots,u_{w_1})\). 
The main result of this paper is the following. 

\begin{theorem}\label{main_thm}
Using the notation above, we have the following\,{\rm :} \\
The power series expansion of the sigma function \(\sigma(u)=\sigma(u_{w_g},u_{w_{g-1}},\cdots,u_{w_1})\) 
attached to the curve \(\mathscr{C}\) defined by \(f(x,y)=0\) of {\rm (\ref{5.01})} has the following properties\,{\rm :}
\begin{equation*}
\begin{aligned}
\sigma(u)&\in\mathbb{Z}[\bvec{\mu}']\llg u_{w_g}, \cdots, u_{w_1}\rrg, \\
\sigma(u)^2&\in\mathbb{Z}[\bvec{\mu}]\llg u_{w_g}, \cdots, u_{w_1}\rrg.
\end{aligned}
\end{equation*}
\end{theorem}
\begin{remark}
(1) To the best knowledge of the author, there are no example 
such that \(\sigma(u)\in\mathbb{Z}[\bvec{\mu}]\llg u_{w_g}, \cdots, u_{w_1}\rrg\). \\ 
(2) Let \(x^iy^j\) be a term of \(f(x,y)\) with both of \(i\) and \(j\) being odd. 
The suffix \({dq-di-qj}\) of its coefficient \(\mu_{dq-di-qj}\) belongs to the Weierstrass gap sequence 
(\ref{w_gaps}) for the pair \((d,q)\). 
\end{remark}
Let us give an illustration. 
\begin{example}
We quote first several terms of \(\sigma(u)\) for \((e,q)=(3,4)\) from Theorem 7.1 in \cite{eemop_2009} 
Dividing the terms by weight with respect to \((u_{w_g}, \cdots, u_{w_1})\) only as 
\begin{equation*}
\sigma(u)=C_5+C_6+C_7+C_8+C_9+\cdots,
  \end{equation*}
then we see
  \begin{equation*}
    \begin{aligned}
C_5&={u_5}-{u_1}{u_2}^2+6\frac{{u_1}^5}{5!},\\
C_6&=2{\mu_1}\frac{{u_1}^4}{4!}\frac{u_2}{1!}-2{\mu_1}\frac{{u_2}^3}{3!},\\
C_7&=10({\mu_1}^2-3{\mu_2})\frac{{u_1}^7}{7!}+2{\mu_2}\frac{{u_1}^3}{3!}\frac{{u_2}^2}{2!},\\
C_8&=2({\mu_1}^3+9{\mu_3}-2{\mu_1}{\mu_2})\frac{{u_1}^6}{6!}\frac{u_2}{1!}-6{\mu_3}\frac{{u_1}^2}{2!}\frac{{u_2}^3}{3!},\\
C_9&=14({\mu_1}^2-3{\mu_2})^2\frac{{u_1}^9}{9!}+2(2{\mu_4}-{\mu_2}^2+{\mu_1}^2{\mu_2}+6{\mu_1}{\mu_3})\frac{{u_1}^5}{5!}\frac{{u_2}^2}{2!}\\
&\qquad -2(4{\mu_1}{\mu_3}+4{\mu_4}+{\mu_2}^2)\frac{u_1}{1!}\frac{{u_2}^4}{4!}+2{\mu_4}\frac{{u_1}^4}{4!}\frac{u_5}{1!},\\
&\qquad\qquad\cdots\cdots\cdot.
    \end{aligned}
  \end{equation*}
Observing these terms, they looks like Hurwitz integral. 
\par
Our main result \ref{main_thm} claims that the expansion is Hurwitz integral over \\
\(\mathbb{Z}[\mu_1, \mu_4, \mu_2, \tfrac12\mu_5, \mu_8, \mu_3, \mu_6, \mu_9,\mu_{12}]\) 
(we need to divide only \(\mu_5\) by \(2\)). 
According to a computation by J.C.~Eilbeck, we have 

\begin{equation*}
\begin{aligned}
\frac{3}{2}{\mu_5}^2\frac{{u_1}^5{u_5}^2}{5!2!} \ \ &[15],\ \ \ 
-\frac{1}{2}{\mu_5}^2\frac{{u_1}{u_2}^2{u_5}^2}{2!2!} \ \ [15],\ \ \
\frac{1}{4}{\mu_5}^2\frac{{u_5}^3}{3!} \ \ [15],\\
-\frac{15}{2}{\mu_2}{\mu_5}^2\frac{{u_1}^7{u_5}^2}{7!2!} \ \ &[17],\ \ \ 
\frac{1}{2}{\mu_2}{\mu_5}^2\frac{{u_1}^3{u_2}^2{u_5}^2}{3!2!2!} \ \ [17],\\
\frac{63}{2}{\mu_2}^2{\mu_5}^2\frac{{u_1}^9{u_5}^2}{9!2!} \ &[19], \ \ \ 
-\frac{1}{2}{u_2}^2{\mu_2}^2{\mu_5}^2\frac{{u_1}^5{u_2}^2{u_5}^2}{5!2!2!} \ \ [19], \ \ \ 
-\frac{1}{2}{\mu_2}^2{\mu_5}^2\frac{{u_1}{u_2}^4{u_5}^2}{2!4!} \ \ [19], \\
\frac{1}{4}{\mu_5}^3\frac{{u_1}^3{u_2}{u_5}^3}{3!3!} \ \ &[20].\\
\end{aligned}
\end{equation*}
Here, each number in \([\ \ ]\) indicates the weight of the term with respect to \(u_{w_j}\)s. 
\end{example}
This paper is organized as follows. 
Using the result in \cite{nakayashiki_2010} of Nakayashiki, 
the sigma function is expressed as a product of the following two functions. 
The first one is a determinant of infinite size whose entries 
are coefficients of functions with only pole at \(\infty\) with respect to
certain local parameter \(t\) at \(\infty\), and 
are indexed by \(\mathbb{N}\times\mathbb{N}\), where \(\mathbb{N}\) is the set of positive integers. 
The other function is an exponential of power series defined by using
some coefficients of expansion with respect to \(t\) of certain differential \(1\)-form of the first kind, 
and of expansion of so-called Klein's differential \(2\)-form for \(\mathscr{C}\),
which is determined from the \(2g\) forms representing a base of \(H^1(\mathscr{C},\mathbb{C})\). 
The former one has zeroes on the same place with \(\sigma(u)\) and 
is expanded at the origin as a power series of Hurwitz integral over \(\mathbb{Z}[\bvec{\mu}]\). 
However, the second derivatives of this determinant are not periodic functions. 
By multiplying the later exponential function, we get periodic functions. 
Such exponential function is expanded as a power series 
not of Hurwitz integral at the prime \(2\) (over \(\mathbb{Z}[\bvec{\mu}]\)). 
So, our proof of the main theorem is divided into two parts corresponding to such two factors of \(\sigma(u)\). 
\vskip 25pt
\section{Arithmetic local parameter}\label{arith_local_param}
In this section, we define a special local parameter called the {\it arithmetic local parameter}. 
Let \(e\) and \(q\) be two coprime positive integers such that \(q>e\). 
Consider the pairs \((a,b)\)  of integers such that
\begin{equation*}
ae-bq=1. 
\end{equation*}
We chose the pair \((a,b)\) such that the absolute value \(|a|\) is minimal in the such pairs. 
Note that, in this case, \(|b|\) is also minimal. 
We define \(\mathrm{sign}(a)\) to be \(1\) or \(-1\) 
according to the sign of \(a\). Let
\begin{equation*}
c=-2a+\mathrm{sign}(a)q, \ \ \ \ 
d=-2b+\mathrm{sign}(a)e.
\end{equation*}
Then we see
\begin{equation}\label{c_and_d}
-ce+dq=2. 
\end{equation}
Especially, we have
\begin{equation*}
|a|<\lceil{q/2}\rceil, \ \ \ |b|<\lceil{e/2}\rceil, \ \ \ 
|c|<\lceil{q/2}\rceil, \ \ \ |d|<\lceil{e/2}\rceil,
\end{equation*}
and that
\begin{equation*}
\begin{aligned}
(ad-bc)e&=d+2b, \ \ \ |d+2b|<\tfrac32(e+1)\,; \\
(ad-bc)q&=c+2a, \ \ \ |c+2a|<\tfrac32(q+1).
\end{aligned}
\end{equation*}
Because of (\ref{c_and_d}), we have
\begin{equation*}
ad-bc=\mathrm{sign}(a). 
\end{equation*}
Therefore, the four integers \(a\), \(b\), \(c\), and \(d\) are simultaneously positive or negative. 
Now for the curve \(\mathscr{C}\) defined by (\ref{plane_miura}), we let
\begin{equation*}
t=x^{-a}y^b, \ \ \ s=x^cy^{-d}. 
\end{equation*}
Then \(t\) is a local parameter at \(\infty\). 
The parameter \(t\) has nice properties as we will see. 
We call this \(t\) the {\it arithmetic local parameter} of \(\mathscr{C}\). 
Obviously, the weights are  \(\wt(t)=1\), \(\wt(s)=2\), and 
\begin{equation}\label{x_y_by_t_s}
x=t^{-|d|}s^{-|b|}, \ \ \ 
y=t^{-|c|}s^{-|a|}.
\end{equation}
Plugging this into (\ref{plane_miura}), we let
\begin{equation}\label{f_tilde}
\begin{aligned}
\tilde{f}(t,s)&=
f\bigg(\frac{1}{t^{|d|}s^{|b|}},\,\frac{1}{t^{|c|}s^{|a|}}\bigg)\,t^{|c|e}s^{|a|e}\cdot
\begin{cases}
t^2 & (a,b,c,d>0)\\
-s   & (a,b,c,d<0)
\end{cases}\\
&=-t^2+s+\cdots\in\mathbb{Z}[\bvec{\mu}][t,s]. 
\end{aligned}
\end{equation}
In particular, \(\mathrm{wt}\big(\tilde{f}(t,s)\big)=2\).
\par
If \(x\) and \(y\) satisfy \(f(x,y)=0\), then 
the corresponding \(t\) and \(s\) satisfy \(\tilde{f}(t,s)=0\). 
Then, using (\ref{f_tilde}) recursively, \(s\) is expressed as a power series of \(t\) such that 
\begin{equation*}
s=t^2+\cdots\in{t^2}\mathbb{Z}[\bvec{\mu}][[t]]. 
\end{equation*}
Moreover, \(x=x(t)\) and  \(y=y(t)\) are expanded as power series of \(t\);
\begin{equation*}
\begin{aligned}
x(t)&=t^{-e}+\cdots\in{t^{-e}}\mathbb{Z}[\bvec{\mu}][[t]], \\
y(t)&=t^{-q}+\cdots\in{t^{-q}}\mathbb{Z}[\bvec{\mu}][[t]].
\end{aligned}
\end{equation*}

\begin{example}
For the case \((e,q)=(3,4)\), we have the following :
  \begin{equation*}
  \begin{aligned}
  f(x,y)&=
  y^3+(\mu_1x+\mu_4)y^2+(\mu_2x^2+\mu_5x+\mu_8)y-(x^4+\mu_3x^3+\mu_6x^2+\mu_9x+\mu_{12}),\\
  \tilde{f}(t,s)&=-t^2+s+(-\mu_3t^3+\mu_1t)s+(-\mu_6t^4+\mu_5t^3+\mu_4t^2)s^2+(-\mu_9t^5+\mu_8t^4)s^3-\mu_{12}t^6s^4,\\
  x(t)&=t^{-3}+\mu_1t^{-2}-\mu_3+\mu_4t+(-\mu_4\mu_1+\mu_5)t^2+(\mu_4{\mu_1}^2-\mu_5\mu_1-\mu_6)t^3+\cdots,\\
  y(t)&=t^{-4}+\mu_1t^{-3}-\mu_3t^{-1}+\mu_4+(-\mu_4\mu_1+\mu_5)t+(\mu_4{\mu_1}^2-\mu_5\mu_1-\mu_6)t^2+\cdots.
  \end{aligned}
  \end{equation*}
\end{example}
We are going to express our differential forms given by \((x,y)\)-coordinates in terms of \((s,t)\). 
We display in increasing order the integers belonging to \(\{\,me+nq\,|\,m,\,n=0,1,2,3,\cdots\}\) as follows:
\begin{equation*}
0=m_0e+n_0q,\ e=m_1e+n_1q,\ \cdots,\ m_je+n_jq,\ \cdots.
\end{equation*}
It is well-known that \(1\), \(2g-1\), and all the integers greater than or equal to \(2g\) appear in this sequence, and
exactly \(g\) integers less than \(2g\) appear. 
Let 
\begin{equation}\label{w_gaps}
w_1(=1), \ \ w_2, \ \cdots, \ \ w_g(=2g-1)
\end{equation}
be the subsequence in increasing order of the integers less than \(2g\). 
These are exactly the Weierstrass gaps at the point \(\infty\) of \(\infty\in\mathscr{C}\). 
For \(j=0\), \(\cdots\), \(g-1\), we define
\begin{equation}\label{omega}
\omega_{w_{g-j}}(t)=\omega_{w_{g-j}}(x,y)=
-\frac{x^{m_j}y^{n_j}dx}{f_y(x,y)}=(t^{w_{g-j}-1}+\cdots)dt\in{t^{w_{g-j}-1}}\mathbb{Z}[\bvec{\mu}][[t]]dt. 
\end{equation}
Then, \(\mathrm{wt}\big(\omega_{w_{g-j}}(t)\big)=w_{g-j}\) (we set \(\mathrm{dt}=1\)), and 
the set of these forms makes a base of the space of the differential forms of the first kind. 
\par
In order to express \(\omega_{w_j}\) in terms of \(t\) and \(s\) algebraically, 
we use notation 
\begin{equation*}
\left\{\begin{matrix}w\\v\end{matrix}\right\}
=\begin{cases}
\,w & (a,b,c,d>0)\\
\,v & (a,b,c,d<0)\\
\end{cases}.
\end{equation*}
From 
\begin{equation*}
f(x,y)=\tilde{f}(t,s)\,t^{-|c|e}s^{-|a|e}
\cdot\left\{\begin{matrix}-t^{-2} \\ s^{-1}\end{matrix}\right\},
\end{equation*}
we see that
\begin{equation*}
\begin{aligned}
f_y(x,,y)
&=\tilde{f}_s(t,s)\,t^{-|c|e+|c|+|d|+1}s^{-|a|e+|a|+|b|+1}\,\frac{dx}{dt}
\cdot\left\{\begin{matrix}t^{-2} \\ s^{-1}\end{matrix}\right\}.
\end{aligned}
\end{equation*}
Therefore, we have
\begin{equation}\label{f_tilde_form}
\frac{dx}{f_y(x,y)}=\frac{t^{|c|e-|c|-|d|-1}s^{|a|e-|a|-|b|-1}}{\tilde{f}_s(t,s)}
\cdot\left\{\begin{matrix} t^2 \\ s\end{matrix}\right\}\,dt. 
\end{equation}
\begin{example}
In the case \((e,q)=(3,4)\), we have
\begin{equation*}
\begin{aligned}
\omega_5(x,y)&=-\frac{dx}{f_y}=\frac{st^2dt}{\tilde{f}_s}\\
&=\big(t^4-2\mu_1t^5+(3{\mu_1}^2-2\mu_2)t^6+(-4{\mu_1}^3+6\mu_2\mu_1+2\mu_3)t^7+\cdots\big)dt,\\
\omega_2(x,y)&=-\frac{xdx}{f_y}=\frac{tdt}{\tilde{f}_s}\\
&=\big(t-\mu_1t^2+{\mu_1}^2t^3+(-{\mu_1}^3+\mu_3)t^4+\cdots\big)dt,\\
\omega_1(x,y)&=-\frac{ydx}{f_y}=\frac{dt}{\tilde{f}_s}\\
&=\big(1-\mu_1t+({\mu_1}^2-\mu_2)t^2+(-{\mu_1}^3+2\mu_2\mu_1+\mu_3)t^3+\cdots\big)dt.\\
\end{aligned}
\end{equation*}
\end{example}
\vskip 25pt
\section{Klein's fundamental 2-form}\label{the_klein_2-form}
In this section we define {\it Klein's fundamental \(2\)-form}, and investigate its Hurwitz integrality. 
Klein's fundamental \(2\)-form is useful for computing a natural symplectic base of the vector space
\begin{equation}\label{h1}
H^1(\mathscr{C},\mathbb{C})\simeq 
\displaystyle\varinjlim_n H^0(\mathscr{C},d\mathscr{O}(n\infty))\big/\varinjlim_n dH^0(\mathscr{C},\mathscr{O}(n\infty))
\end{equation}
(see \cite{nakayashiki_2008} or \cite{lect_chuo}). 
\begin{definition}\label{3.2}{\rm
Take two variable points  \((x,y)\) and \((z,w)\) on \(\mathscr{C}\). 
For the arithmetic local parameter \(t\), let \(t_1\) and \(t_2\) be 
its values at these points, and define
\begin{equation}\label{bracket}
[t_1,t_2]dt_1=-\omega_{w_g}(t_1)\frac{1}{x-z}\frac{f(Z,y)-f(Z,w)}{y-w}\Big|_{Z=z},
\end{equation}
where  \(Z\) is an indeterminate.
}
\end{definition}
\begin{proposition}\label{def_eta}
We use the notation {\rm (\ref{bracket})}. 
For \(i=1\), \(\cdots\), \(g\), there exists a form of the second kind \(\eta_{-w_i}(t)\) of weight \(-w_i\) 
in \(\dfrac{dt}{t^{w_i+1}}\mathbb{Z}[\bvec{\mu}][[t]]\) satisfying, for 
\begin{equation*}
\bvec{\xi}(t_1,t_2)=\frac{d}{dt_2}[t_1,t_2]dt_1dt_2+\sum_{i=1}^g\omega_{w_i}(t_1)\,\eta_{-w_i}(t_2), 
\end{equation*}
the following two equations:
\begin{align}
\bvec{\xi}(t_1,t_2)
&\in\frac{dt_1dt_2}{(t_2-t_1)^2}+\mathbb{Z}[\bvec{\xi}][[t_1,t_2]]dt_1dt_2, \label{klein}\\
\bvec{\xi}&(t_1,t_2)=\bvec{\xi}(t_2,t_1). \label{klein2}
\end{align}
\(\bvec{\xi}(t_1,t_2)\) is homogeneous of weight \(2\). 
Those \ \(\eta_{-w_i}(t)\) \ are not unique. 
\end{proposition}
We call \ \(\bvec{\xi}(t_1,t_2)\) \ above {\it Klein's fundamental \(2\)-form}. 
\begin{remark}
Nakayashiki \cite{nakayashiki_2008} constructed Klein's fundamental \(2\)-form analytically 
by using prime forms which are expressed in terms of first derivatives of a theta function, 
following \cite{fay_1973}. 
However, in view of our aim, we shall construct it algebraically. 
\end{remark}
\begin{proof} 
We fix an primitive \(e\)-th root of \(1\) and denote it by \(1^{1/e}\). 
Take the pull-back images of the function \(t\mapsto{x(t)}\), and denote them by  \(t,\ti{1},\ti{2},\cdots,\ti{e-1}\). 
Precisely speaking, we assume that each \(\ti{j}\) has value such that \(x(t)=x(\ti{j})\) for the given \(t\) and satisfies
\begin{equation*}
\ti{j}=1^{1/e}\,t+\cdots\in\mathbb{Q}[[t]]. 
\end{equation*}
Using the notation of Definition \ref{3.2}, we define
\begin{equation*}
{[t_1,\,t_2]}
=-\frac{\omega_{2g-1}(x,y)}{dt_1}\,
 \frac{1}{x-z}\frac{f(Z,y)-f(Z,w)}{y-w}\Big|_{Z=z}.
\end{equation*}
Then
\begin{equation*}
\lim_{t_2\to t_1}\frac{t_1-t_2}{x-z}\frac{f(Z,y)-f(Z,w)}{y-w}\Big|_{Z=z}
=\frac{f_y(x,y)}{\frac{d}{dt}x(t_1)}. 
\end{equation*}
So that, we see
\begin{equation}\label{lim_1}
\lim_{t_2\to t_1}(t_1-t_2){[t_1,\ t_2]}=1.
\end{equation}
By Weierstra{\ss} preparation theorem, we have
\begin{equation*}
x(t_1)^{-1}-x(t_2)^{-1}
=(t_1-t_2)(t_1-\tj{1}{2})(t_1-\tj{2}{2})\cdots(t_1-\tj{e-1}{2})\,p(t_1,t_2)
\end{equation*}
with some  \(p(t_1,t_2)\in1+(t_1,t_2)\mathbb{Z}[\bvec{\mu}][[t_1,t_2]]\). 
Therefore, the expansions with respect to \(t_2\) of  \(\tj{1}{2}\), \(\cdots\), \(\tj{e-1}{2}\) are 
all belong to \(\mathbb{Z}[[t_2]]\). 
Moreover, we have
\begin{equation*}
\begin{aligned}
 \frac{1}{x-z}&\frac{f(Z,y)-f(Z,w)}{y-w}\Big|_{Z=z}\\
&= \frac{-(xz)^{-1}}{x^{-1}-z^{-1}}\frac{f(Z,y)-f(Z,w)}{y-w}\Big|_{Z=z}\\
&=\frac{1}{t_1-t_2}\cdot
\frac{y(t_1)-y(\tj{1}{2})}{t_1-\tj{1}{2}}\cdot
\frac{y(t_1)-y(\tj{2}{2})}{t_1-\tj{2}{2}}\cdots
\frac{y(t_1)-y(\tj{e-1}{2})}{t_1-\tj{e-1}{2}}\\
&\hskip 200pt\cdot\,(-x(t_1)x(t_2))^{-1}\,p(t_1,p_2)^{-1}. 
\end{aligned}
\end{equation*}
Here the middle \(e-1\) factors above should be reduced by Weierstra{\ss} preparation theorem. 
This means that, for \(j=1\), \(2\), \(\cdots\), \(e{-}1\), 
\begin{equation*}
\frac{y(t_1)-y(\tj{j}{2})}{t_1-\tj{j}{2}}
\in\,\bigcup_{r=1}^q{t_1}^{-r}({\tj{j}{2}})^{-q+r-1}\,\mathbb{Z}[\bvec{\mu}][[t_1,\tj{j}{2}]].
\end{equation*}
Consequently, we see
\begin{equation}\label{expansion1}
(t_2-t_1)\frac{1}{x-z}\frac{f(Z,y)-f(Z,w)}{y-w}\Big|_{Z=z}
\in\,{t_1}^e{t_2}^e\bigcup_{r=e-1}^{q(e-1)}{t_1}^{-r}{t_2}^{-(q-1)(e-1)+r}\,\mathbb{Z}[\bvec{\mu}][[t_1,\tj{j}{2}]].
\end{equation}
Since
\begin{equation*}
\omega_{2g-1}(t_1)\in{t_1}^{(q-1)(e-1)-1}\mathbb{Z}[\bvec{\mu}][[t_1]]dt_1,
\end{equation*}
(\ref{lim_1}) and (\ref{expansion1}) imply that
\begin{equation}\label{expansion2}
\begin{aligned}
(t_2-t_1)&\frac{\omega_{2g-1}(t_1)}{dt_1}\frac{1}{x-z}\frac{f(Z,y)-f(Z,w)}{y-w}\Big|_{Z=z}\\
&\in\,1+(t_2-t_1)\!\!\bigcup_{r=-1}^{(q-1)(e-1)+1}{t_1}^{r}{t_2}^{2-r}\,\mathbb{Z}[\bvec{\mu}][[t_1,\tj{j}{2}]].
\end{aligned}
\end{equation}
The space of differential form on \(\mathscr{C}\) having unique pole only at \(\infty\) at most of order \((w_g+1)\)
are spanned by the forms of first kinds and some other \(w_g\) forms of the second kind. 
Obviously, those are of the form
\begin{equation}\label{diff_forms}
\frac{x^ky^{\ell}}{f_y(x,y)}dx \ \ \ \ (k,\ \ell=0,\,1,\,2,\cdots). 
\end{equation}
\par
From the result above, 
\(\frac{d}{dt_2}[t_1,\ t_2]dt_1dt_2\), as a form of \(t_1\), has 
pole at \(t_2\) with essential part \(1/(t_1-t_2)^2\), 
and as a form of \(t_2\), has only poles at \(t_1\) and \(0\). 
Accordingly, if we subtract from \(\frac{d}{dt_2}[t_1,\ t_2]dt_1dt_2\) a \(2\)-form which is a bi-linear combination 
\begin{equation}\label{substractions}
\sum_{i=1}^g\omega_{w_i}(t_1)\,\eta_{-w_i}(t_2)
\end{equation}
of forms in (\ref{diff_forms}) over \(\mathbb{Z}[\bvec{\mu}]\)
and some forms \(\eta_{-w_i}(t)\) of the second kind with a pole at \(t=0\) of degree \(-(w_i-1)\)) and no other poles, 
we get a \(2\)-form which has, as a form of \(t_1\), 
a pole at \(t_2\) with its essential part \((t_1-t_2)^{-2}dt_1dt_2\) and no other poles, 
and has, as a form of \(t_2\), the same type pole at \(t_1\) and no other poles. 
Moreover, its coefficients of the expansion should be belong to \(\mathbb{Z}[\bvec{\mu}]\) 
by our argument above. Therefore, it belongs to
\begin{equation*}
\frac{dt_1dt_2}{(t_1-t_2)^2}+\mathbb{Z}[\bvec{\mu}][[t_1,t_2]]dt_1dt_2. 
\end{equation*}
Since the 2-form  \(\frac{d}{dt_2}[t_1,t_2]dt_1dt_2\)  is of homogeneous weight 
with respect to  \(\{\mu_j\}\), \(t_1\), and \(t_2\),  
the 2-form  (\ref{substractions}) should be of homogeneous weight. 
\par
Now, we define a square matrix \(M\) of size \(w_g+w_{g-1}+\cdots+w_1\) as follows. 
For each \(i=g\), \(\cdots\), \(1\), we take \(w_i\) numbers 
\(j=g+w_i-1\), \(g+w_i-2\), \(\cdots\), \(g\), and make the pairs \(\{(i,j)\}\) be the numbering of columns of \(M\). 
For each  \(k=g\), \(\cdots\), \(1\), we take \(w_k-2\) numbers 
\(\ell=-w_k-1\), \(-w_k\), \(\cdots\), \(-2\), and make the pairs \((k,\ell)\) be the numbering of lows of  \(M\). 
We define
\begin{equation*}
M=\big[\,C_{i,j;k,\ell}\,\big],
\end{equation*}
where  \(C_{i,j;k,\ell}\) is the coefficient of \({t_1}^{w_k-1}\,{t_2}^{-\ell}\) in the expansion of 
\(\omega_{w_i}(t_1)\frac{x^{m_j}y^{n_j}}{f_y}\frac{dx}{dt}(t_2)\). 
If we take square matrices, say \(M_j\), of size \(w_j\) diagonally from  \(M\) successively \(w_j\), 
the entries below of these matrices are all \(0\), and each \(M_j\) is a lower triangular matrix. \\
(The proof is continued.)
\end{proof}
We shall illustrate above argument by an example. 
It will be better the rest of proof also is explained through this example. 
\begin{example}
If  \((e,q)=(3,4)\), then
\begin{equation*}
\begin{aligned}
{[}t_2,\ &t_1{]} + \frac1{t_2-t_1}\\
=&-\frac{1}{t_1}\\
&+\big(
{t_1}^4
-2{\mu_1}\,{t_1}^5+\cdots
\big)\frac{1}{{t_2}^5}\\
&+\big(
2{\mu_1}\,{t_1}^4
-4{\mu_1}^2{t_1}^5+\cdots)\frac{1}{{t_2}^4}\\
&+\big(
({\mu_1}^2+2{\mu_2}){t_1}^4
-(2{\mu_1}^3+4{\mu_2}{\mu_1}){t_1}^5+\cdots
\big)\frac{1}{{t_2}^3}\\
&+\big(
{t_1}
-{\mu_1}\,{t_1}^2
+({\mu_1}^2-{\mu_2}){t_1}^3
+(-{\mu_1}^3+4{\mu_2}{\mu_1}){t_1}^4\\
&\hskip 70pt
+({\mu_1}^4-7{\mu_2}{\mu_1}^2-2{\mu_4}+{\mu_2}^2){t_1}^5+\cdots
\big)\frac{1}{{t_2}^2}\\
&+\big(
1
-{\mu_2}\,{t_1}^2
+({\mu_2}{\mu_1}+{\mu_3}){t_1}^3
+(-{\mu_2}{\mu_1}^2-2{\mu_3}{\mu_1}+2{\mu_2}^2){t_1}^4\\
&\hskip 70pt 
+({\mu_2}{\mu_1}^3+3{\mu_3}{\mu_1}^2-4{\mu_2}^2{\mu_1}-2{\mu_3}{\mu_2}-2{\mu_5}){t_1}^5+\cdots
\big)\frac{1}{t_2}\\
&+\big(
{\mu_1}
+(-{\mu_1}^2+{\mu_2}){t_1}
+({\mu_1}^3-2{\mu_2}{\mu_1}){t_1}^2+\cdots
\big)\\
&+\big(
-{\mu_4}\,{t_1}^2
+(2{\mu_4}{\mu_1}-{\mu_5}){t_1}^3
+(-3{\mu_4}{\mu_1}^2+2{\mu_5}{\mu_1}+2{\mu_2}{\mu_4}){t_1}^4+\cdots
\big){t_2}\\
&+\cdots.
\end{aligned}
\end{equation*}
Differentiating this with respect to \(t_2\), we have
\begin{equation}\label{diff_t2}
\begin{aligned}
\frac{d}{dt_2}&{[t_1,t_2]}-\frac1{(t_2-t_1)^2}\\
&=\Big(-5{t_1}^4
+10{\mu_1}\,{t_1}^5+\cdots
\Big)\frac{1}{{t_2}^6}\\
&+\Big(
-8{\mu_1}\,{t_1}^4
+16{\mu_1}^2\,{t_1}^5+\cdots
\Big)\frac{1}{{t_2}^5}\\
&+\Big(
-(3{\mu_1}^2+6{\mu_2}){t_1}^4
+(6{\mu_1}^3+12{\mu_2}{\mu_1}){t_1}^5+\cdots
\Big)\frac{1}{{t_2}^4}\\
&+\Big(
-2\,{t_1}^1
+2{\mu_1}\,{t_1}^2
+(-2{\mu_1}^2+2{\mu_2}){t_1}^3
+(2{\mu_1}^3-8{\mu_2}{\mu_1}){t_1}^4\\
&\ \ \ \ +(-2{\mu_1}^4+14{\mu_2}{\mu_1}^2-2{\mu_2}^2+4{\mu_4}){t_1}^5+\cdots
\Big)\frac{1}{{t_2}^3}\\
&+\Big(
-{t_1}^0
+({\mu_2}){t_1}^2
-({\mu_2}{\mu_1}+{\mu_3}){t_1}^3
+({\mu_2}{\mu_1}^2+2{\mu_3}{\mu_1}-2{\mu_2}^2){t_1}^4\\
&\ \ \ \ +(-{\mu_2}{\mu_1}^3-3{\mu_3}{\mu_1}^2+4{\mu_2}^2{\mu_1}+2{\mu_3}{\mu_2}+2{\mu_5}){t_1}^5+\cdots
\Big)\frac{1}{{t_2}^2}\\
&+\Big(
-{\mu_4}\,{t_1}^2
+(2{\mu_4}{\mu_1}-{\mu_5}){t_1}^3
+(-3{\mu_4}{\mu_1}^2+2{\mu_5}{\mu_1}+2{\mu_4}{\mu_2}){t_1}^4\\
&\ \ \ \ \ +(4{\mu_4}{\mu_1}^3-3{\mu_5}{\mu_1}^2-6{\mu_4}{\mu_2}{\mu_1}+2{\mu_5}{\mu_2}-2{\mu_4}{\mu_3}){t_1}^5+\cdots
\Big)\\
&+\cdots.
\end{aligned}
\end{equation}
\par
The forms in a base of the space of forms of the second kind and their expansions with respect to 
the arithmetic parameter \(t\) with \((x,y)=(x(t),y(t))\) are as follows:
\begin{equation}\label{simplicity}
\begin{aligned}
\frac{x^2y}{f_y(x,y)}dx&=\Big(\frac{-1}{t^6}-\frac{\mu_1}{t^5}-\frac{\mu_2}{t^4}+\frac{\mu_3}{t^3}+(-{\mu_4}^2+{\mu_8})t^2+\cdots\Big)\,dt\,\in\frac{1}{t^6}\mathbb{Z}[\bvec{\mu}][[t]]dt,\\
\frac{x^3}{f_y(x,y)}dx&=\Big(\frac{-1}{t^5}-\frac{\mu_1}{t^4}-\frac{\mu_2}{t^3}+\frac{\mu_3}{t^2}+(-{\mu_4}^2+{\mu_8})t^3+\cdots\Big)\,dt\,\in\frac{1}{t^5}\mathbb{Z}[\bvec{\mu}][[t]]dt,\\
\frac{y^2}{f_y(x,y)}dx&=\Big(\frac{-1}{t^4}+{\mu_4}+(-2{\mu_4}{\mu_1}+{\mu_5})t+\cdots \Big)\,dt\,\in\frac{1}{t^4}\mathbb{Z}[\bvec{\mu}][[t]]dt,\\
\frac{xy}{f_y(x,y)}dx&=\Big(-\frac{-1}{t^3}+{\mu_4}t+(-2{\mu_4}{\mu_1}+{\mu_5})t^2+\cdots  \Big)\,dt\,\in\frac{1}{t^3}\mathbb{Z}[\bvec{\mu}][[t]]dt,\\
\frac{x^2}{f_y(x,y)}dx&=\Big(-\frac{1}{t^2}+{\mu_4}t^2+\cdots  \Big)\,dt\,\in\frac{1}{t^2}\mathbb{Z}[\bvec{\mu}][[t]]dt. 
\end{aligned}
\end{equation}
\end{example}
\par
We want, by choosing suitable  \(\{\,a_j,\ b_j,\ c_j\,\}\) such that
\begin{equation}\label{abc}
\begin{aligned}
\eta_{-5}(t)&=
\frac{
a_0x^2y
+a_1x^3
+a_2y^2
+a_3xy
+a_4x^2
}{f_y(x,y)}dx,\\
\eta_{-2}(t)&=
\frac{
b_0y^2
+b_1xy
}{f_y(x,y)}dx,\\
\eta_{-1}(t)&=
\frac{
c_0xy
}{f_y(x,y)}dx
\end{aligned}
\end{equation}
satisfy (\ref{klein}). 
Here each of \(a_j\), \(b_j\), and \(c_j\) is an element in \(\mathbb{Z}[\bvec{\mu}]\) of 
weight \(-j\) or equals to \(0\). 
These elements are obtained as a set of solutions of some linear equation 
whose coefficient matrix is \(M\) in the proof. 
The matrix \(M\) is given by the following table. 
\\
{\scriptsize
\begin{equation*}
\hskip -50pt
\begin{array}{|c|ccccc|cc|c|}
\hline\\[-1.3em]
\struth
\multirow{2}*{\mbox{\begin{minipage}{65pt}{\strutHigh\ \ \ expansion in \(t_1\)\\\(\times\) expansion in \(t_2\)}\end{minipage}}} 
                               & \multicolumn{5}{c|}{\frac{1}{f_y}\frac{dx}{dt}(t_1)\,\cdot}     & \multicolumn{2}{c|}{\frac{ x}{f_y}\frac{dx}{dt}(t_1)\cdot}  & \frac{ y}{f_y}\frac{dx}{dt}(t_1)\cdot \\[0pt]
                               & \frac{x^2y}{f_y}\frac{dx}{dt}(t_2) & \frac{x^3}{f_y}\frac{dx}{dt}(t_2) & \frac{y^2}{f_y}\frac{dx}{dt}(t_2) & \frac{xy}{f_y}\frac{dx}{dt}(t_2) & \frac{x^2}{f_y}\frac{dx}{dt}(t_2) & \frac{xy}{f_y}\frac{dx}{dt}(t_2) & \frac{x^2}{f_y}\frac{dx}{dt}(t_2) & \frac{xy}{f_y}\frac{dx}{dt}(t_2) \\[5pt]\hline\\[-1.34em]
\struth\mbox{coeff. of \ \({{t_1}^4}{{t_2}^{-6}}\)}\ \hfill &                1 &               0 &               0 &              0 &               0 & \strutHigh                 0 &                            0 &                                                     0 \\[0pt]
\struth\mbox{coeff. of \ \({{t_1}^4}{{t_2}^{-5}}\)}\,\hfill &            \mu_1 &               1 &               0 &              0 &               0 &                            0 &                            0 &                                                     0 \\[0pt]
\struth\mbox{coeff. of \ \({{t_1}^4}{{t_2}^{-4}}\)}\,\hfill &            \mu_2 &           \mu_1 &               1 &              0 &               0 &                            0 &                            0 &                                                     0 \\[0pt]
\struth\mbox{coeff. of \ \({{t_1}^4}{{t_2}^{-3}}\)}\,\hfill &           -\mu_3 &           \mu_2 &               0 &              1 &               0 & -{\mu_1}^3{+}2\mu_2\mu_1{+}\mu_3 &                            0 &                                                     0 \\
\struth\mbox{coeff. of \ \({{t_1}^4}{{t_2}^{-2}}\)}\,\hfill &                0 &          -\mu_3 &               0 &              0 &               1 &                            0 & -{\mu_1}^3{+}2\mu_2\mu_1{+}\mu_3 & \mbox{\scriptsize\(\bigg\{\!\!\!\begin{array}{l}{\mu_1}^4{-}3\mu_2{\mu_1}^2\\ \ -2\mu_3\mu_1{-}2\mu_4+{\mu_2}^2\end{array}\)} \\[6pt]\cline{2-8}\\[-1.35em]
\struth\mbox{coeff. of \ \({{t_1}  }{{t_2}^{-3}}\)}\,\hfill &                0 &               0 &               0 &              0 & \strutHigh    0 &                            1 &                            0 &                                                     0 \\
\struth\mbox{coeff. of \ \({{t_1}  }{{t_2}^{-2}}\)}\,\hfill &                0 &               0 &               0 &              0 &               0 &                            0 &                            1 &                                                -\mu_1 \\[0pt]\cline{7-9}\\[-1.35em]
\struth\mbox{coeff. of \ \({{1\,}  }{{t_2}^{-2}}\)}\,\hfill &                0 &               0 &               0 &              0 &               0 &                            0 &                            0 &                                                     1 \\[0pt]
\hline
\end{array}
\end{equation*}}
This table is given as follows. 
For example, we have \(\mu_2\) in the \((4,2)\)-entry.  
This is the coefficient of  \lqq\({t_1}^4{t_2}^{-3}\) in the expansion 
of  \(\frac{1}{f_y}\frac{dx}{dt}(t_1)\) times \(\frac{x^3}{f_y}\frac{dx}{dt}(t_2)\). 
This means that \(\mu_2\) is the coefficient \(-1\) of \({t_1}^4\) in the expansion of \(\frac{1}{f_y}\frac{dx}{dt}(t_1)\) 
times the coefficient \(-\mu_2\) of  \(\frac{1}{{t_2}^4}\) in the expansion of \(\frac{x^3}{f_y}\frac{dx}{dt}(t_2)\). 
\par
Extracting this table as a \(8\times 8\) matrix, we denote it \(M\). 
In general  \(M\) is a square matrix od size  \(w_g+w_{g-1}+\cdots+w_1\). 
\begin{lemma}
The matrix \(M\) belong to the ring of square matrices of size \(w_g+w_{g-1}+\cdots+w_1\) with entries in \(\mathbb{Z}[\bvec{\mu}]\), 
and it is a unit in that ring. 
\end{lemma}
\begin{proof}
The matrix \(M\) is diagonally concatenated by several lower triangular matrices \(\{M_j\}\)
(they are indeed lower triangular because those entries are the coefficients of 
the expansion with respect to \(t_2\) of the corresponding 1-forms of the second kind which are \(0\))
and the entries below the \(M_j\)s are all \(0\)
(because those entries are the coefficients of the expansion with respect to 
\(t_1\) of the corresponding 1-forms of the first kind which are \(0\)). 
Concluding from the argument above, we have proved the lemma. 
\end{proof}
Here, we shall list up first several coefficients from (\ref{diff_t2}):
\begin{equation*}
\bvec{q}
=\left[\ 
\begin{matrix}
\mbox{the coefficient of}\ \frac{{t_1}^4}{{t_2}^6}\\[3pt]
\mbox{the coefficient of}\ \frac{{t_1}^4}{{t_2}^5}\\[3pt]
\mbox{the coefficient of}\ \frac{{t_1}^4}{{t_2}^4}\\[3pt]
\mbox{the coefficient of}\ \frac{{t_1}^4}{{t_2}^3}\\[3pt]
\mbox{the coefficient of}\ \frac{{t_1}^4}{{t_2}^2}\\[3pt]
\mbox{the coefficient of}\ \frac{{t_1}  }{{t_2}^3}\\[3pt]
\mbox{the coefficient of}\ \frac{{t_1}  }{{t_2}^2}\\[3pt]
\mbox{the coefficient of}\ \frac{ 1     }{{t_2}^2}
\end{matrix}\ \right]
=\left[\ 
\begin{matrix}
-5\\[2.5pt]
-8\mu_1\\[2.5pt]
-(3\mu_1^2+6\mu_2)\\[2.5pt]
2(\mu_1^3-4\mu_2\mu_1)\\[2.5pt]
\mu_2{\mu_1}^2+2\mu_3\mu_1-2{\mu_2}^2\\[2.5pt]
-2\\[5pt]
0\mu_1\\[5pt]
-1\\
\end{matrix}
\ \right]
\end{equation*}
Since 
\begin{equation*}
\begin{aligned}
\frac{d}{dt_2}{[t_1,t_2]}dt_1dt_2-&\frac{dt_1dt_2}{(t_2-t_1)^2}
\ \in\ -\omega_5(t_1)\,\frac{a_0{x^2y}+a_1{x^3}+a_2{y^2}+a_3{xy}+a_4{x^2}}{f_y(x,y)}dx(t_2)\\
&-\omega_2(t_1)\,\frac{                            b_0{xy}+b_1{x^2}}{f_y(x,y)}dx(t_2)
-\omega_1(t_1)\,\frac{                                    c_0{x^2}}{f_y(x,y)}dx(t_2)
+\mathbb{Z}[\bvec{\mu}][[t_1,t_2]],
\end{aligned}
\end{equation*}
we have
\begin{equation*}
\left[\ 
\begin{matrix}
a_0\\ a_1\\ a_2\\ a_3\\ a_4\\ b_0\\ b_1\\ c_0
\end{matrix}\ \right]
=M^{-1}\bvec{q}
=
\left[\,
\begin{matrix}
5\\ 3\mu_1\\ \mu_2\\  \mu_2\mu_1+3\mu_3\\  2\mu_3\mu_1+{\mu_2}^2+2\mu_4\\ 2\\ \mu_1\\ 1
\end{matrix}\,
\right].
\end{equation*}
So, we obtained 
\begin{equation*}
\begin{aligned}
\eta_{-5}(t)&=\frac{\ 5{x^2y}+3\mu_1{x^3}+\mu_2{y^2}+\mu_3{xy}+2\mu_1\mu_3{x^2}\ }{f_y(x,y)}\,dx(t)\\
            &=(5t^{-6}+6{\mu_2}t^{-4}-2{\mu_3}t^{-3}+{\mu_2}^2t^{-2}+\mu_6+\cdots)dt,\\
\eta_{-2}(t)&=\frac{\                                    2{xy}+      \mu_1{x^2}\ }{f_y(x,y)}\,dx(t)\\
            &=(2t^{-3}-2{\mu_5}t^2+\cdots)dt,\\
\eta_{-1}(t)&=\frac{\                                                      x^2 \ }{f_y(x,y)}\,dx(t)\\
            &=(t^{-2}-{\mu_5}t^3+\cdots)dt. 
\end{aligned}
\end{equation*}
\begin{proof} \ ({\it continuation.}) 
In the next place, we shall modify the obtained 
\begin{equation*}
\bvec{\xi}(t_1,t_2)=\frac{d}{dt_2}[t_1,t_2]dt_1dt_2+\sum_{i=1}^g\omega_{w_i}(t_1)\,\eta_{-w_i}(t_2)
\in\frac{dt_1dt_2}{(t_2-t_1)^2}+\mathbb{Z}[\bvec{\xi}][[t_1,t_2]]dt_1dt_2
\end{equation*}
in a symmetric form with respect to the exchange for \(t_1\) and \(t_2\). 
For the above \(\bvec{\xi}\), the 2-form on \(\mathscr{C}\)
\begin{equation*}
\bvec{\xi}(t_1,t_2)-\bvec{\xi}(t_2,t_1)
\end{equation*}
is holomorphic because the term \(1/(t_1-t_2)^2\)  is canceled out. 
Therefore, it can be written as
\begin{equation*}
\sum_{i,j=1}^g a_{i,j}\,\omega_{w_i}(t_1)\,\omega_{w_j}(t_2), 
\end{equation*}
where \(a_{i,j}\)s are in \(\mathbb{Z}[\vec]\) with \(a_{i,j}=-a_{j,i}\) (\(\forall i\), \(\forall j\)). 
Especially, \(a_{i,i}=0\). 
Then, since
\begin{equation*}
\begin{aligned}
 \bigg(\bvec{\xi}(t_1,t_2)&-\sum_{i>j}a_{i,j}\,\omega_{w_i}(t_1)\,\omega_{w_j}(t_2)\bigg)
 -\bigg(\bvec{\xi}(t_2,t_1) -\sum_{i>j}a_{i,j}\,\omega_{w_i}(t_2)\,\omega_{w_j}(t_1)\bigg)\\
&=\bigg(\bvec{\xi}(t_1,t_2)-\sum_{i>j}a_{i,j}\,\omega_{w_i}(t_1)\,\omega_{w_j}(t_2)\bigg)
 -\bigg(\bvec{\xi}(t_2,t_1) -\sum_{j>i}a_{j,i}\,\omega_{w_j}(t_2)\,\omega_{w_i}(t_1)\bigg)\\
&=\sum_{i=1}^g a_{i,i}\omega_{w_i}(t_1)\,\omega_{w_i}(t_2)=0, 
\end{aligned}
\end{equation*}
we redefine \(\xi(t_1,t_2)\) as the 2-form
\begin{equation*}
\bvec{\xi}(t_1,t_2)-\sum_{i>j}a_{i,j}\,\omega_{w_i}(t_1)\,\omega_{w_j}(t_2). 
\end{equation*}
This means if we replace \(\eta_{-w_i}(t_2)\) for \(i=1\), \(\cdots\), \(g\) by
\begin{equation*}
\eta_{-w_i}(t_2)-\sum_{i>j}a_{i,j}\,\omega_{w_j}(t_2),
\end{equation*}
all the desired conditions are satisfied. 
\end{proof}
\vskip 10pt
For instance, in the \((3,4)\)-curve case, 
it is necessary to replace only \ \(\eta_{-5}(t)\) \ by 
\begin{equation*}
\begin{aligned}
\eta_{-5}(t)&=\big(\,5{x^2y}+3\mu_1{x^3}+\mu_2{y^2}+\mu_3{xy}+2\mu_1\mu_3{x^2}\\
&\qquad\qquad +(\mu_6+\mu_4\mu_2)y+(\mu_1\mu_6+\mu_2\mu_5+\mu_3\mu_4)x\,\big)\frac{1}{f_y(x,y)}\,dx(t). 
\end{aligned}
\end{equation*}
The finally obtained Klein's fundamental 2-form is 
\begin{equation*}
\begin{aligned}
\bvec{\xi}&((x,y),(z,w))
   =\frac{F(x,y,z,w)dxdz}{(x-z)^2f_y(x,y)f_y(z,w)}\\
  &=\Bigg(\frac{1}{(t_2-t_1)^2}
   +\mu_4{t_1}^2
   +2\mu_4t_2t_1
   +\mu_4{t_2}^2\\
  &\qquad
   +(-2\mu_4\mu_1+\mu_5)({t_1}^3+{t_2}^3)
   +(-4\mu_4\mu_1+2\mu_5)(t_2{t_1}^2+{t_2}^2t_1)
   +\cdots\Bigg)\,dt_1dt_2, 
\end{aligned}
\end{equation*}
where
  \begin{align*}\label{1.03}
  F(&x,y,z,w)=
   3{}y^2{}w^2-2{}(x{}y{}z^3+x^3{}z{}w)+(x^2{}z^2{}w+x^2{}y{}z^2)
   +({\mu_8}{}{\mu_4}+3{}{\mu_{12}}){}(y+w)\\
  &+{\mu_8}{}(y^2+w^2)+({\mu_{12}}{}{\mu_1}+{\mu_9}{}{\mu_4}+{\mu_8}{}{\mu_5}){}(x+z)
   +({\mu_8}{}{\mu_1}+{\mu_9}){}(x{}y+z{}w)\\
  &+(2{}{\mu_4}^2-2{}{\mu_8}){}y{}w
   +({\mu_5}{}{\mu_4}+2{}{\mu_9}){}(y{}z+w{}x)
   +(2{}{\mu_4}{}{\mu_1}-{\mu_5}){}(y{}x{}w+y{}z{}w)\\
  &+({\mu_4}{}{\mu_2}+{\mu_6}){}(y{}z^2+w{}x^2)
   +2{}{\mu_4}{}(y^2{}w+y{}w^2)
   +{\mu_5}{}(y^2{}z+x{}w^2)\\
  &+2{}{\mu_1}{}(x{}y{}w^2+y^2{}z{}w)
   +{\mu_2}{}(y^2{}z^2+x^2{}w^2)
   +(2{}{\mu_9}{}{\mu_1}+2{}{\mu_8}{}{\mu_2}+2{}{\mu_4}{}{\mu_6}+{\mu_5}^2){}x{}z\\
  &+({\mu_5}{}{\mu_1}+2{}{\mu_6}){}(x{}z{}w+x{}y{}z)
   +({\mu_6}{}{\mu_1}+{\mu_5}{}{\mu_2}+{\mu_4}{}{\mu_3}){}(x{}z^2+x^2{}z)\\
  &+(2{}{\mu_1}^2-2{}{\mu_2}){}x{}y{}z{}w
   +(2{}{\mu_3}{}{\mu_1}+{\mu_2}^2+2{}{\mu_4}){}x^2{}z^2
   +({\mu_2}{}{\mu_1}+3{}{\mu_3}){}(x{}y{}z^2+x^2{}z{}w)\\
  &+{\mu_1}{}(x^2{}z^3+x^3{}z^2)
   +({\mu_8}^2+2{}{\mu_{12}}{}{\mu_4}).
  \end{align*}
\begin{remark}\label{inner_prod}
We denote by \(\mathscr{C}^{\circ}\) the {\it regular polygon} 
obtained by cutting off  \(\mathscr{C}\) by a set of \(g\) paths passing through \(\infty\)
which represents a symplectic base of \(H_1(\mathscr{C},\mathbb{Z})\).  
Let \(\omega\) and  \(\eta\) are any two 1-forms of the second kind with no poles elsewhere \(\infty\), 
and we regard them as elements in \(H^1(\mathscr{C},\mathbb{C})\) via (\ref{h1}). 
As usual, the space \(H^1(\mathscr{C},\mathbb{C})\) is equipped with the inner product defined by
\begin{equation*}
\omega\star\eta=\int_{\partial \mathscr{C}^{\circ}}\left(\int_{\infty}^{\mathrm{P}}\omega(\mathrm{P})\right)\eta(\mathrm{P})
=\sum_{\mathscr{C}^{\circ}}\Res\left(\int_0^{t}\omega(t)\right)\eta(t). 
\end{equation*}
Then the set of \(\{\omega_{w_j}\}\) and \(\{\eta_{-w_j}\}\) gives 
a symplectic base of the space \(H^1(\mathscr{C},\mathbb{C})\) (see \cite{nakayashiki_2010}). 
Indeed, if \(\omega\), \(\eta\in\{\omega_{w_j},\,\eta_{-w_j}\,|\,j=1,\,\cdots\,g\}\), we have
\begin{equation*}
\omega\star\eta=
\Res\limits_{t=0}\left(\int_{\mathrm{formal}}^{t}\omega(t)\right)\eta(t)
\end{equation*}
and symplectic property follows from the expansion we obtained before. 
\end{remark}
\begin{example}
Here, we shall show the integrals for the case \((e,q)=(3,4)\):
\begin{equation*}
\begin{aligned}
\left(\int_{\mathrm{formal}}^{t}\!\!\!\!\omega_5(t)\right)\eta_{-5}(t)&=(t^{-1}-\tfrac{8}{35}{\mu_2}t+\cdots)dt, \ 
\left(\int_{\mathrm{formal}}^{t}\!\!\!\!\omega_5(t)\right)\eta_{-2}(t) =(\tfrac{2}{5}t^2-\tfrac{4}{7}{\mu_2}t^4+\cdots)dt,\\
\left(\int_{\mathrm{formal}}^{t}\!\!\!\!\omega_5(t)\right)\eta_{-1}(t)&=(\tfrac{1}{5}t^3-\tfrac{2}{7}{\mu_2}t^5+\cdots)dt,\\
\left(\int_{\mathrm{formal}}^{t}\!\!\!\!\omega_2(t)\right)\eta_{-5}(t)&=(\tfrac{5}{2}t^{-4}+\tfrac{7}{4}{\mu_2}t^{-2}-\tfrac{1}{6}{\mu_2}^2+\cdots)dt,\\
\left(\int_{\mathrm{formal}}^{t}\!\!\!\!\omega_2(t)\right)\eta_{-2}(t)&=(t^{-1}-\tfrac{1}{2}\mu_2t+\cdots)dt,\
\left(\int_{\mathrm{formal}}^{t}\!\!\!\!\omega_2(t)\right)\eta_{-1}(t) =(\tfrac{1}{2}-\tfrac{1}{4}{\mu_2}t^2+\cdots)dt,\\
\left(\int_{\mathrm{formal}}^{t}\!\!\!\!\omega_1(t)\right)\eta_{-5}(t)&=(5t^{-5}+\tfrac{13}{3}\mu_2t^{-3}-\tfrac{3}{4}\mu_3t^{-2}-(-\tfrac{1}{2}\mu_3\mu_2+\tfrac{5}{3}\mu_5)+\cdots)dt,\\
\left(\int_{\mathrm{formal}}^{t}\!\!\!\!\omega_1(t)\right)\eta_{-2}(t)&=(2t^{-2}-\tfrac{2}{3}\mu_2+\cdots)dt,\ 
\left(\int_{\mathrm{formal}}^{t}\!\!\!\!\omega_1(t)\right)\eta_{-1}(t)=(t^{-1}-\tfrac{1}{3}{\mu_2}t+\cdots)dt.
\end{aligned}
\end{equation*}
\end{example}
\vskip 25pt 
\section{The standard theta cycles}\label{theta_cycles}
We fix here some notations including {\it standard theta cocycles}. 
Firstly we define the discriminant of \(\mathscr{C}\). 
\begin{definition}\label{discriminant}
Assuming all the \(\mu_j\)s are indeterminates, we define
  \begin{equation*}
  \begin{aligned}
  R_1&=\rslt_x
  \Big(\rslt_y\big(f(x,y), \tfrac{\partial}{\partial x}f(x,y)\big), 
       \rslt_y\big(f(x,y), \tfrac{\partial}{\partial y}f(x,y)\big)
  \Big), \\
  R_2&=\rslt_y
  \Big(\rslt_x\big(f(x,y), \tfrac{\partial}{\partial x}f(x,y)\big), 
       \rslt_x\big(f(x,y), \tfrac{\partial}{\partial y}f(x,y)\big)
  \Big),\\
  R&=\gcd(R_1,R_2)\ \ \ \ \mbox{in \(\mathbb{Z}[\bvec{\mu}]\)}. 
  \end{aligned}
  \end{equation*}
Here  \(\rslt_z\)  stands for taking the resultant of Sylvester. 
Then \(R\) is undoubted a square element in the ring \(\mathbb{Z}[\bvec{\mu}]\).
However, the author does not have any proof of this. 
We can check this by computer for the cases
\((d,q)=(2,3)\), \((2,5)\), \((3,4)\). 
No proof for this is not a defect for our context. 
We denote by \(D\) one of the square root of \(R\). 
We take the signature of this square root with adding (\ref{c}) 
as the equality \ref{inf_det} exactly holds. 
After all these operation, in order to get the correct {\it discriminant} \(D\), 
we substitute the own value of \(\mu_j\) to \(D\). 
The reason for assuming \(\{\mu_j\}\) be indeterminates is because 
taking greatest common divisor and substitution for \(\{\mu_j\}\) some values are not commutative. 
\end{definition}

If all \(\mu_j\)s belong to \(\mathbb{C}\) and the discriminant \(D\) of \(\mathscr{C}\) is not \(0\), then 
\begin{equation*}
\varLambda=\Bigg\{\oint(\omega_{w_1},\cdots,\omega_{w_g})\ \ \mbox{for any closed paths on \(\mathscr{C}\)}\,\Bigg\}
\subset\mathbb{C}^g
\end{equation*}
is a lattice of the space \(\mathbb{C}^g\). 
Then for each \(k\geqq 1\), we define {\it Abel-Jacobi map} by
\begin{equation*}
\begin{aligned}
\iota : \mathrm{Sym}^k\mathscr{C} &\longrightarrow J,\\
(\P_1,\cdots,\P_k)&\mapsto \sum_{j=1}^k\int_{\infty}^{\P_j}(\omega_{w_1},\cdots,\omega_{w_g})\ \bmod{\varLambda}.
\end{aligned}
\end{equation*}
Its image is denoted by \(W\br{k}\). 
We assume \(W\br{0}\) is a one point set consists of the origin of \(J\). 
Moreover, we define
\begin{equation}\label{theta_cycle}
\Theta\br{k}=W\br{k}\cup [-1]W\br{k}, \ \ \ (k\geqq 0),
\end{equation}
where  \([-1]\) is an involution which transform \(u_{w_1},\cdots,u_{w_g}\) to \(-u_{w_1},\cdots,-u_{w_g}\). 
We know that \(W\br{k}=\Theta\br{k}\) for \(k\geq g-1\), and 
\(W\br{k}=\Theta\br{k}=J\) for \(k\geqq g\). 
We call \(\Theta\br{k}\) (\(0\leqq k\leqq g-1\)) the {\it standard theta cocycles} and call \(\Theta\br{g-1}\) 
the {\it standard theta divisor} of \(J\).
\vskip 25pt
\section{Definition of the sigma function}\label{def_sigma}
In this section, we define the sigma function precisely. 
For proofs, the reader is referred to \cite{lect_chuo}. 
\par
We have already defined the 2-forms \(\{\eta_{w_j}\}\) of second kind. 
Let introduce four \(g\times g\) matrices
\begin{equation*}
  \begin{aligned}
\omega'&=\bigg[\int_{\alpha_j}\omega_{w_{g-i+1}}\bigg]\ , \ \ 
\omega''=\bigg[\int_{\beta_j}\omega_{w_{g-i+1}}\bigg]\ , \\
\eta'&=\bigg[\int_{\alpha_j}\eta_{w_{g-i+1}}\bigg]\ , \ \ 
\eta''=\bigg[\int_{\beta_j}\eta_{w_{g-i+1}}\bigg].
  \end{aligned}
\end{equation*}
For an arbitrary \(u\in\mathbb{C}^g\), we define \(u'\) and \(u''\in\mathbb{R}^g\) by
\begin{equation}\label{prime_primes}
u=\omega'u'+\omega''u''.
\end{equation}
We write as \({\omega'}^{-1}\tp{(\omega_{w_g},\cdots,\omega_{w_1})}=\tp{(\hat{\omega}_1,\cdots,\hat{\omega}_g)\)}, 
\({\omega'}^{-1}\omega''=\big[T_{ij}\big]\), and define
\begin{equation*}
\begin{aligned}
\delta_j&=-\tfrac12T_{jj}-\int_{\infty}^{\I_j}\hat{\omega}_j
+\sum_{i=1}^g\int_{\alpha_i}\bigg(\int_{\infty}^{\P}\hat{\omega}_j\bigg)\,
\hat{\omega}_i(\P),\\
\delta&={\omega'}\,\tp{(\delta_1,\cdots,\delta_g)}.
\end{aligned}
\end{equation*}
Here, the integral are given on the regular polygon \(\mathscr{C}_{\circ}\) 
(defined in \ref{inner_prod}) whose boundary 
is \(\alpha_1\circ{\alpha_1}^{-1}\circ\beta_1\circ{\beta_1}^{-1}\circ\cdots\circ\alpha_g\circ{\alpha_g}^{-1}\circ\beta_g\circ{\beta_g}^{-1}\) 
after regarding  \(\mathscr{C}\) as a compact Riemann surface, 
and \(\I_j\) is the initial point of  \(\alpha_j\) corresponding \(\infty\). 
The vector \(\delta\) is the {\it Riemann constant} determined by 
\(\mathscr{C}\) with our base point \(\infty\). 
Since the divisor of \(\omega_{w_g}\) is \((2g-2)\infty\), we have
\begin{equation*}
\delta\in\tfrac12\Lambda. 
\end{equation*}
Following the definition (\ref{prime_primes}), 
we define \(\delta'\) and  \(\delta''\in\frac12\mathbb{Z}^g\) by \(\delta=\omega'\delta'+\omega''\delta''\). 
For any \(u\), \(v\in\mathbb{C}^g\) and any \(\ell\in\Lambda\), we define
\begin{equation*}
\begin{aligned}
L(u,v)&=\tp{u}(v'\eta'+v''\eta''), \\
\chi(\ell)&=\exp\big(\,2\pi\i(\ell'\,\tp{\delta'}+\ell''\,\tp{\delta''}+\tfrac12\ell'\,\tp{\ell''})\,\big). 
\end{aligned}
\end{equation*}
Then a theorem of Frobenius shows that the entire functions \(\sigma(u)\) on the space \(\mathbb{C}^g\) 
satisfying the {\it translational relation}
\begin{equation}\label{translational}
\sigma(u+\ell)=\chi(\ell)\,\sigma(u)\,\exp\,L(u+\tfrac12\ell,\ell), \ \ \ (u\in\mathbb{C}^g, \ \ \ell\in\Lambda)
\end{equation}
form a \(1\)-dimensional vector space. 
Such functions has zeroes of order \(1\) exactly along the pull-back image of  \(\Theta\br{g-1}\) via 
the canonical map \(\kappa\) : \(\mathbb{C}^g\rightarrow\mathbb{C}^g/\Lambda=J\). 
\par
We shall describe this in other words. 
There exists a vector of functions
\begin{equation*}
  \zeta(u)=[\zeta_{w_1}(u), \cdots, \zeta_{w_g}(u)]
\end{equation*}
having properties 
  \begin{equation*}
  \zeta(u+\ell)=\zeta(u)+\eta'\ell'+\eta''\ell'' \ \ \ (\,\ell\in\Lambda\,)
  \end{equation*}
and that there exists unique entire function \(\sigma(u)\) on  \(\mathbb{C}^g\), 
up to multiplication of a constant, 
which has zeroes of order \(1\) exactly along \(\kappa^{-1}(\Theta\br{g-1})\) and satisfies
  \begin{equation*}
  d\log\sigma(u)=\zeta_{w_1}(u)du_{w_1}+\zeta_{w_2}(u)du_{w_2}+\cdots+\zeta_{w_g}(u)du_{w_g}\ (\,=\zeta(u)du\,).
  \end{equation*}
\vskip 10pt
\noindent
\textbf{Construction of the sigma function via a theta series}\\
We define a function \(\sigma(u)\) by
  \begin{equation*}
  \sigma(u)=c\,\exp\big(-\tfrac12\tp{u}{\eta'}{\omega'}^{-1}u\big)
  \,\vartheta\bigg[\begin{matrix}\delta''\\\delta'\end{matrix}\bigg]
  ({\omega'}^{-1}u;{\omega'}^{-1}\omega''),
  \end{equation*}
and call it the {\it sigma function} for  \(\mathscr{C}\). 
Here, \(c\neq0\) is a constant independent of \(u\) is not important in this paper. 
We define \(c\) as a constant by which the expression \ref{inf_det} holds.  
The correct value of \(c\) might given by
  \begin{equation}\label{c}
  c=\frac{1}{D^{1/8}}\bigg(\frac{|\omega'|}{(2\pi)^g}\bigg)^{1/2},
  \end{equation}
where  \(\pi=3.1415\cdots\) is the circle ratio, \(|\omega'|\) is the determinant of the matrix \(\omega'\), 
and \(D\) is the discriminant defined by \ref{discriminant}. 
However, the choice of the signature of the right hand side in (\ref{c}) 
is not clear for the author. 
\par 
If we choose another symplectice base of \(H_1(\mathscr{C},\mathbb{Z})\) 
the period integral \(\omega'\), \(\omega''\), \(\eta'\), \(\eta''\) \ and 
the Riemann constants \(\delta'\), \(\delta''\) are transformed.  
But \(\sigma(u)\) itself is invariant under this change, which property is proved 
by using modular transformation property of the theta series
({see \cite{baker_1897}, p.536}). 
Of course, the translational formula for the theta series with respect to the change  \(u\) by \(u+\ell\)
implies (\ref{translational}). 
\vskip 25pt
\section{The determinantal expression of the sigma function}\label{the_det_expression}
Now we proceed to explain the determinantal expression of the sigma function. 
We use the arithmetic parameter \(t\) in place of the local parameter \(z\) used in \cite{nakayashiki_2010}. 
Moreover we denote the set of infinitely many number of variables \(\{t_j\}\) used in \cite{nakayashiki_2010} 
by \(\{U_j\}\), which is used in the theory of tau functions in integrable systems. 
We write as\footnote{Please do not confuse with \(c_j\) in \ref{abc}.}
\begin{equation}\label{c_j}
\log \left(\sqrt{\frac{1}{t^{2g-2}}\frac{\omega_{w_g}(t)}{dt}}\,\right)
=\sum_j \frac{c_j}jt^j. 
\end{equation}
\begin{equation}
\omega_{w_g}(t)\in{t^{w_g}}\,\mathbb{Z}[\bvec{\mu}][[t]]dt
\end{equation}
So that, 
\begin{equation*}
c_j\in\tfrac12\mathbb{Z}[\bvec{\mu}]. 
\end{equation*}
It is important for our main theorem that, for which \(j\), we have \(c_j\in\mathbb{Z}[\bvec{\mu}]\).  
\par
We define \(q_{ij}\in\mathbb{Z}[\bvec{\mu}]\) by
\begin{equation*}
\bvec{\xi}(t_1,t_2)-\frac{dt_1dt_2}{(t_1-t_2)^2}
=\sum_{i,j}q_{ij}{t_1}^{i-1}{t_2}^{j-1}dt_1dt_2. 
\end{equation*}
Using \(\{q_{ij}\}\), we set
\begin{equation*}
q(u)=\sum_{j=1}^{g}c_{w_j}U_{w_j}+\frac12\sum_{i=1}^g\sum_{j=1}^gq_{ij}U_{w_i}U_{w_j}. 
\end{equation*}
Since \(\frac12{U_{w_i}}^2=\frac{1}{2!}{U_{w_i}}^2\) and  \(q_{ij}=q_{ji}\in\mathbb{Z}[\bvec{\mu}]\), 
the later sum is Hurwitz integral with respect to \(\{U_j\}\). 
For a monomial \(\varphi(t)\) of  \(x(t)\) and \(y(t)\), we write its coefficient as 
\begin{equation*}
\varphi(t)=\sum_{j}(\varphi)_{(j)}\,t^j. 
\end{equation*}
We display all the monomials of \(x(t)\) and \(y(t)\) in increasing order of the order of the pole at \(t=0\). 
We define a matrix \(\varGamma\) with numbering by integers in lows and doing by non-negative integers by 
\begin{equation}\label{varGamma}
\varGamma=\left[\ 
\begin{matrix}
\ddots & \vdots                  & \vdots                  & \vdots              \\
\cdots & \varphi_{v_3}(t)_{(-6)} & \varphi_{v_2}(t)_{(-6)} & \varphi_{v_1}(t)_{(-6)} & \varphi_0(t)_{(-6)} \\
\cdots & \varphi_{v_3}(t)_{(-5)} & \varphi_{v_2}(t)_{(-5)} & \varphi_{v_1}(t)_{(-5)} & \varphi_0(t)_{(-5)} \\
\cdots & \varphi_{v_3}(t)_{(-4)} & \varphi_{v_2}(t)_{(-4)} & \varphi_{v_1}(t)_{(-4)} & \varphi_0(t)_{(-4)} \\
\cdots & \varphi_{v_3}(t)_{(-3)} & \varphi_{v_2}(t)_{(-3)} & \varphi_{v_1}(t)_{(-3)} & \varphi_0(t)_{(-3)} \\
\hline
\cdots & \varphi_{v_3}(t)_{(-2)} & \varphi_{v_2}(t)_{(-2)} & \varphi_{v_1}(t)_{(-2)} & \varphi_0(t)_{(-2)} \\
\cdots & \varphi_{v_3}(t)_{(-1)} & \varphi_{v_2}(t)_{(-1)} & \varphi_{v_1}(t)_{(-1)} & \varphi_0(t)_{(-1)} \\
\cdots & \varphi_{v_3}(t)_{( 0)} & \varphi_{v_2}(t)_{( 0)} & \varphi_{v_1}(t)_{( 0)} & \varphi_0(t)_{( 0)} \\
\cdots & \varphi_{v_3}(t)_{( 1)} & \varphi_{v_2}(t)_{( 1)} & \varphi_{v_1}(t)_{( 1)} & \varphi_0(t)_{( 1)} \\
\cdots & \varphi_{v_3}(t)_{( 2)} & \varphi_{v_2}(t)_{( 2)} & \varphi_{v_1}(t)_{( 2)} & \varphi_0(t)_{( 2)} \\
\ddots & \vdots                  & \vdots                  & \vdots                  & \vdots              
\end{matrix}
\ \right].
\end{equation}
Let \(T\) be a new indeterminate, set
\begin{equation*}
\begin{aligned}
\bvec{T}&=U_1T+U_2T^2+U_3T^3+U_4T^4+U_5T^5+\cdots.\\
(&\mbox{For the (\(3,4\))-curve, this will be set }\  =U_1T+U_2T^2+U_5T^5\ \ \mbox{later}.)
\end{aligned}
\end{equation*}
We define  \(p_j\in\mathbb{Q}[\bvec{\mu}]\llg{U}\rrg\)  by
  \begin{equation*}
  \begin{aligned}
  \bvec{p}&=1+\frac1{1!}\bvec{T}+\frac1{2!}\bvec{T}^2+\frac1{3!}\bvec{T}^3+\frac1{4!}\bvec{T}^4+\frac1{5!}\bvec{T}^5+\cdots\\
         &=p_0+p_1T+p_2T^2+p_3T^3+p_4T^4+p_5T^5+\cdots.
  \end{aligned}
  \end{equation*}
Then, we see 
\begin{equation*}
p_j\in\mathbb{Z}[\bvec{\mu}]\llg{U}\rrg\cap\mathbb{Q}[\bvec{\mu}][U].
\end{equation*}
Using those matrices, we define 
\begin{equation}\label{S}
S(U)=S(B^{-1}u)
 =\left[\ 
 \begin{matrix}
 \ddots & \vdots\ & \vdots & \vdots & \vdots & \vdots & \vdots\!\!\!\!&\vertbar&\!\! \vdots & \vdots & \vdots & \vdots & \vdots  & \ddots\\
 \cdots &   1   \ &  p_1   &  p_2   &  p_3   &  p_4   &  p_5  \!\!\!\!&\vertbar&\!\!  p_6   &  p_7   &  p_8   &  p_9   &  p_{10} & \cdots\\
 \cdots &   0   \ &    1   &  p_1   &  p_2   &  p_3   &  p_4  \!\!\!\!&\vertbar&\!\!  p_5   &  p_6   &  p_7   &  p_8   &  p_9    & \cdots\\
 \cdots &   0   \ &    0   &    1   &  p_1   &  p_2   &  p_3  \!\!\!\!&\vertbar&\!\!  p_4   &  p_5   &  p_6   &  p_7   &  p_8    & \cdots\\
 \cdots &   0   \ &    0   &    0   &    1   &  p_1   &  p_2  \!\!\!\!&\vertbar&\!\!  p_3   &  p_4   &  p_5   &  p_6   &  p_7    & \cdots\\
 \cdots &   0   \ &    0   &    0   &    0   &    1   &  p_1  \!\!\!\!&\vertbar&\!\!  p_2   &  p_3   &  p_4   &  p_5   &  p_6    & \cdots\\
 \cdots &   0   \ &    0   &    0   &    0   &    0   &    1  \!\!\!\!&\vertbar&\!\!  p_1   &  p_2   &  p_3   &  p_4   &  p_5    & \cdots
 \end{matrix}\ 
 \right]. 
 \end{equation}
We explain the main result of \cite{nakayashiki_2010}. 
Let 
\begin{equation*}
B=\left[\ 
\begin{matrix}
\omega_{w_g}(t)_{(0)} & \omega_{w_g}(t)_{(1)} & \omega_{w_g}(t)_{(2)} & \cdots \\
\vdots                & \vdots                & \vdots                & \ddots \\
\omega_{w_2}(t)_{(0)} & \omega_{w_2}(t)_{(1)} & \omega_{w_2}(t)_{(2)} & \cdots \\
\omega_{w_1}(t)_{(0)} & \omega_{w_1}(t)_{(1)} & \omega_{w_1}(t)_{(2)} & \cdots \\
\end{matrix}
\ \right].
\end{equation*}
We need the following 
\begin{theorem}\label{inf_det} \ {\rm (Nakayashiki \cite{nakayashiki_2010})} 
We have
\begin{equation*}
\sigma(BU)
=\det\big(S(U){\varGamma}\big)\cdot\exp\,{q(U)}. 
\end{equation*}
Here, the product of matrices is operated as the lines in {\rm(\ref{varGamma})} and in {\rm(\ref{S})} engage. 
We define {\rm(}the sign of\,{\rm)} the overall multiplicative constant \(c\) of \(\sigma(u)\) by this equality. 
\end{theorem}
\begin{example}
In the case \((e,q)=(3,4)\), if all the \(\mu_j=0\), then \(\sigma(u)\) is given by
\begin{equation*}
\sigma(u_5,u_3,u_1)\big|_{\vec{\mu}=0}=\left|\ 
\begin{matrix}
p_3 & p_4 & p_5 \\
  1 & p_1 & p_2 \\
    &   1 & p_1 \\
\end{matrix}
\ \right|=u_5-u_1{u_2}^2+\tfrac{1}{20}{u_1}^5. 
\end{equation*}
\end{example}
We explain the weight of \(\sigma(u)\). 
\begin{proposition}
The expansion of  \(\sigma(u)\) around the origin is of homogeneous weight. Its weight is 
\begin{equation*}
\mathrm{wt}\,\big(\sigma(u)\big)=\tfrac1{24}(e^2-1)(q^2-1). 
\end{equation*}
\end{proposition}
\begin{proof} 
(Due to \cite{bl_2004}, p.97) 
For our Weierstrass sequence (\ref{w_gaps}) 
\((w_1,w_2,\cdots)\) is the ascending set of positive integers 
that are not representable in the form \(ae+bq\) with non-negative integers \(a\) and \(b\). 
Set \(w(T)=\sum_iT^{w_i}\) with an indeterminate \(T\). 
We have
\begin{equation*}
w(T) =\frac{1}{1-T}-\frac{1-T^{eq}}{(1-T^e)(1-T^q)}.
\end{equation*}
Therefore the length of the Weierstrass sequence \(g=w(1)\), 
and the sum  \(G=\frac{d}{dT}w(1)\)  of the elements of this sequence are given by the formulae
\begin{equation*}
g=\frac{(e-1)(q-1)}2, \ \ G=\frac{eq(e-1)(q-1)}4-\frac{(e^2-1)(q^2-1)}{12}. 
\end{equation*}
Let us introduce new variables \(\{u^{(1)},\cdots,u^{(g)}\}\), with weight \(1\) for all, such that 
the relations \(u_j=\tfrac1{j}\displaystyle\sum_{i=1}^g{u^{(i)}}^j\) holds 
for \(j=w_1\), \(\cdots\), \(w_g\) (Newton's symmetric polynomial). 
Over any algebraically closed field, by \cite{macdonald_1995} p.29, \(\ell.-4\) and p.28, \(\ell\).13 for example, 
we see that for any value of \((u_{w_g},\cdots,u_{w_1})\) there exists a solution of the system of that relations, 
and we know that
\begin{equation*}
\sigma(u_{w_g},u_{w_{g-1}},\cdots,u_{w_1})\big|_{\vec{\mu}=0}
=\frac{\mathrm{det}\,\big[{u^{(i)}}^{w_j}\big]}{\mathrm{det}\,\big[{u^{(i)}}^{j-1}\big]}\ \ 
\ \ (i, j=1, \cdots, g). 
\end{equation*}
Hence, 
\begin{equation*}
\mathrm{wt}\,\big(\sigma(u_{w_g},\cdots,u_{w_1})\big)
=\sum_{k=1}^gw_k-\sum_{k=1}^g(k-1)
=G-\frac{(g-1)g}2
=\frac{(e^2-1)(q^2-1)}{24}.
\end{equation*}
\end{proof}

Using Theorem \ref{inf_det}, we prove that the expansion  \(\sigma(u)\) around the origin is of Hurwitz integral. 
We let denote by \(B_0\) the matrix obtained by removing the columns except for 
\(j=1\), \(\cdots\), \(g\) from \(B\). 
The matrix \(B_0\) is an upper triangular matrix belongs to \(\mathrm{SL}(g,\mathbb{Z}[\bvec{\mu}])\) 
with all its diagonal entries being all \(1\). 
We set all  \(U_i=0\) except for \(U_{w_j}\) (\(j=1\), \(\cdots\), \(g\)), and 
set, for the variables  \(\{u_{w_j}\}\)  of \(\sigma(u)\), 
\begin{equation}\label{Key}
[\ U_{w_g} \ \cdots \ U_{w_2} \ U_{w_1}\,]={B_0}^{-1}\,\tp{[\,u_{w_g} \ \cdots \ u_{w_2} \ u_{w_1}\,]}={B_0}^{-1}u. 
\end{equation}
Then  \(U_{w_j}\in\mathbb{Z}[\bvec{\mu}][u]\) and the following expression of the sigma function by Theorem \ref{inf_det}.
\begin{theorem}
The sigma function is expressed as
\begin{equation}\label{det_expression}
\begin{aligned}
\sigma(u)
=\det\big(S({B_0}^{-1}u){\varGamma}\big)\cdot\exp\,{q({B_0}^{-1}u)}.
\end{aligned}
\end{equation}
\end{theorem}
This is the main tool to prove our main theorem. 
Since \(p_j\) is Hurwitz integral as a polynomial of \(\{u_{w_i}\}\),  
we see \(\det\big(S({B_0}^{-1}u){\varGamma}\big)\) is Hurwitz integral. 
The crucial part of the proof is on the expansion of \(\exp\,{q({B_0}^{-1}u)}\). 
\begin{example}
For the \((3,4)\)-curve, the set \(\{\varphi_{v_i}\}\)  consists of 
\begin{equation*}
  \cdots,\ x^3, \ y^2, \ xy, \  x^2, \ y, \  x,\  1.
\end{equation*}
Taking their expansion with respect to \(t\), we define 
\begin{equation*}
\begin{aligned}
\varGamma
&=
\left[\ \begin{matrix}
\ddots & \vdots      & \vdots      & \vdots     & \vdots       & \vdots     & \vdots     & \vdots \\
\cdots & (x^3)_{(-4)} & (y^2)_{(-4)} & (xy)_{(-4)} &  (x^2)_{(-4)} &  (y)_{(-4)} &  (x)_{(-4)} & (1)_{(-4)} \\
\cdots & (x^3)_{(-3)} & (y^2)_{(-3)} & (xy)_{(-3)} &  (x^2)_{(-3)} &  (y)_{(-3)} &  (x)_{(-3)} & (1)_{(-3)} \\
\hline
\cdots & (x^3)_{(-2)} & (y^2)_{(-2)} & (xy)_{(-2)} &  (x^2)_{(-2)} &  (y)_{(-2)} &  (x)_{(-2)} & (1)_{(-2)} \\
\cdots & (x^3)_{(-1)} & (y^2)_{(-1)} & (xy)_{(-1)} &  (x^2)_{(-1)} &  (y)_{(-1)} &  (x)_{(-1)} & (1)_{(-1)} \\
\cdots & (x^3)_{(0)} & (y^2)_{(0)} & (xy)_{(0)} &  (x^2)_{(0)} &  (y)_{(0)} &  (x)_{(0)} & (1)_{(0)} \\
\cdots & (x^3)_{(1)} & (y^2)_{(1)} & (xy)_{(1)} &  (x^2)_{(1)} &  (y)_{(1)} &  (x)_{(1)} & (1)_{(1)} \\
\cdots & (x^3)_{(2)} & (y^2)_{(2)} & (xy)_{(2)} &  (x^2)_{(2)} &  (y)_{(2)} &  (x)_{(2)} & (1)_{(2)} \\
\ddots & \vdots      & \vdots      & \vdots     & \vdots       & \vdots     & \vdots     & \vdots
\end{matrix}\ \right].
\end{aligned}
\end{equation*}
As we explained before, let define \(c_j\)s by
\begin{equation*}
\begin{aligned}
\log\sqrt{\frac{1}{t^{6-2}}\frac{\omega_5(t)}{dt}}&=
\sum_j \frac{c_j}jt^j \\
&=-\mu_1t+({\mu_1}^2-2\mu_2)\tfrac12t^2
 +(-{\mu_1}^3+3\mu_2\mu_1+3\mu_3)\tfrac13t^3 \\
&\ \ \ \ 
 +({\mu_1}^4-4\mu_2{\mu_1}^2-4\mu_3\mu_1-6\mu_4+2{\mu_2}^2)\tfrac14t^4\\
&\ \ \ \ 
 +(-{\mu_1}^5+5\mu_2{\mu_1}^3+5\mu_3{\mu_1}^2+5(3\mu_4-{\mu_2}^2)\mu_1-5\mu_3\mu_2-\tfrac{15}2\mu_5)\tfrac15t^5\\
&\ \ \ \ 
 +\cdots.
\end{aligned}
\end{equation*}
Then
\begin{equation}\label{c_j2}
\left\{
\begin{aligned}
c_1&=-\mu_1,\\
c_2&={\mu_1}^2-2\mu_2,\\
c_5&=-{\mu_1}^5+5\mu_2{\mu_1}^3+5\mu_3{\mu_1}^2+5(3\mu_4-{\mu_2}^2)\mu_1-5\mu_3\mu_2-\mbox{\fbox{\(\tfrac{15}2\mu_5\)}}.
\end{aligned}\right.
\end{equation}
The boxed last term gives arise some denominator. 
The elements \(q_j\in\mathbb{Z}[\vec{\mu}]\) are defined by
\begin{equation*}
\begin{aligned}
\bvec{\xi}&(t_1,t_2)-\frac{dt_1dt_2}{(t_1-t_2)^2} \\
&=\sum_{i,j}q_{ij}{t_1}^{i-1}{t_2}^{j-1}\\
&={\mu_4}{t_1}^2
+2{\mu_4}{t_2}{t_1}
+{\mu_4}{t_2}^2\\
&\quad +({\mu_5}-2{\mu_4}{\mu_1}){t_1}^3
+(2{\mu_5}-4{\mu_4}{\mu_1}){t_2}{t_1}^2
+(2{\mu_5}-4{\mu_4}{\mu_1}){t_2}^2{t_1}
+({\mu_5}-2{\mu_4}{\mu_1}){t_2}^3\\
&\quad +(3{\mu_4}{\mu_1}^2-2{\mu_5}{\mu_1}-2{\mu_2}{\mu_4}-{\mu_6}){t_1}^4
+(6{\mu_4}{\mu_1}^2-4{\mu_5}{\mu_1}-4{\mu_2}{\mu_4}-2{\mu_6}){t_2}^3{t_1}\\
&\quad +(7{\mu_4}{\mu_1}^2-5{\mu_5}{\mu_1}-4{\mu_2}{\mu_4}-3{\mu_6}){t_2}^2{t_1}^2
+(6{\mu_4}{\mu_1}^2-4{\mu_5}{\mu_1}-4{\mu_2}{\mu_4}-2{\mu_6}){t_2}{t_1}^3\\
&\quad + \cdots,
\end{aligned}
\end{equation*}
and the necessary elements among them are given explicitly by
\begin{equation*}
\left\{
\begin{aligned}
q_{11}&=0, \\
q_{21}&=0, \\
q_{12}&=0, \\
q_{22}&=-2\mu_4, \\
q_{51}&=-3\mu_4{\mu_1}^2+2\mu_5\mu_1+(2\mu_2\mu_4+\mu_6),\\
q_{15}&=-3\mu_4{\mu_1}^2+2\mu_5\mu_1+(2\mu_2\mu_4+\mu_6),\\
q_{52}&=8\mu_4{\mu_1}^3-6\mu_5{\mu_1}^2+(-12\mu_2\mu_4-4\mu_6)\mu_1+(-4\mu_3\mu_4+4\mu_5\mu_2),\\
q_{25}&=8\mu_4{\mu_1}^3-6\mu_5{\mu_1}^2+(-12\mu_2\mu_4-4\mu_6)\mu_1+(-4\mu_3\mu_4+4\mu_5\mu_2),\\
q_{55}&=-23\mu_4{\mu_1}^6+21\mu_5{\mu_1}^5+(96\mu_2\mu_4+19\mu_6){\mu_1}^4\\
      &\ \ +(60\mu_3\mu_4-68\mu_5\mu_2){\mu_1}^3+(56{\mu_4}^2-94{\mu_2}^2\mu_4-44\mu_6\mu_2-38\mu_3\mu_5-34\mu_8){\mu_1}^2\\
      &\ \ +((-72\mu_3\mu_2-44\mu_5)\mu_4+40\mu_5{\mu_2}^2-20\mu_6\mu_3-17\mu_9)\mu_1\\
      &\ \ -24\mu_2{\mu_4}^2+(12{\mu_2}^3-9{\mu_3}^2-12\mu_6)\mu_4+(11\mu_6{\mu_2}^2+19\mu_3\mu_5+18\mu_8)\mu_2+5{\mu_5}^2.\\
\end{aligned}
\right.
\end{equation*}
Setting all \(U_j\) to be \(0\) except \(U_1\), \(U_2\), and \(U_5\). 
Then the relation with the variables of \(\sigma(u)=\sigma(u_5,u_2,u_1)\) is given by
\begin{equation*}
\begin{alignedat}{4}
U_1&=u_1+\mu_1&u_2&+{}&(\mu_2{\mu_1}^2+{\mu_3}\mu_1+2\mu_4-{\mu_2}^2)&u_5,\\
U_2&=         &u_2&+{}&(                 {\mu_1}^3-2\mu_2\mu_1-\mu_3)&u_5,\\
U_5&=         &   & {}&                                              &u_5.
\end{alignedat}
\end{equation*}
Using those elements, we have
\begin{equation*}
\begin{aligned}
q(U)&=q({B_0}^{-1}\,\tp{[\,u_5,u_2,u_1\,]})\\
   &=c_1U_1+c_2U_2+c_5U_5\\
   &\qquad-\tfrac12(q_{11}{U_1}^2+(q_{12}+q_{21})U_2U_1+(q_{15}+q_{51})U_1U_5\\
   &\qquad\quad +q_{22}{U_2}^2+(q_{52}+q_{25})U_5U_2+q_{55}{U_5}^2). 
\end{aligned}
\end{equation*}
In our case, 
\begin{equation*}
\begin{aligned}
p_1&={u_1}+{\mu_1}{u_2}+({\mu_2}{\mu_1}^2+{\mu_3}{\mu_1}+2{\mu_4}-{\mu_2}^2){u_5},\\
p_2&={u_2}+\tfrac12{u_1}^2+\tfrac12{\mu_1}^2{u_2}^2+({\mu_1}^3-2{\mu_2}{\mu_1}-{\mu_3}){u_5}\\
&\qquad +{\mu_1}{u_2}{u_1}+({\mu_2}{\mu_1}^2+{\mu_3}{\mu_1}+2{\mu_4}-{\mu_2}^2){u_5}{u_1}\\
&\qquad\ +({\mu_2}{\mu_1}^3+{\mu_3}{\mu_1}^2+2{\mu_4}{\mu_1}-{\mu_2}^2{\mu_1}){u_5}{u_2}\\
&\qquad\ \  +\big(
          \tfrac12{\mu_2}^2{\mu_1}^4+{\mu_3}{\mu_2}{\mu_1}^3+(2{\mu_2}{\mu_4}-{\mu_2}^3+\tfrac12{\mu_3}^2){\mu_1}^2\\
&\qquad\ \ \  +(2{\mu_3}{\mu_4}-{\mu_3}{\mu_2}^2){\mu_1}+2{\mu_4}^2-2{\mu_2}^2{\mu_4}+\tfrac12{\mu_2}^4\big){u_5}^2,\\
p_3&=\tfrac16{u_1}^3+u_2u_1+\mbox{\lq\lq higher weight terms in \(u_1\), \(u_2\), \(u_5\)"},\\
p_4&=\tfrac1{24}{u_1}^4+\tfrac12u_2{u_1}^2+\tfrac12{u_2}^2+\mbox{\lq\lq higher weight terms in \(u_1\), \(u_2\), \(u_5\)"},\\
p_5&=\tfrac1{120}{u_1}^5+\tfrac16u_2{u_1}^3+\tfrac12{u_2}^2u_1+u_5+\mbox{\lq\lq higher weight terms in \(u_1\), \(u_2\), \(u_5\)"},
 \end{aligned}
\end{equation*}
and so on. 
\end{example}
The rest of this paper is devoted to investigate 
how does the prime \(2\) occur in denominators, in general, like the denominator \(2\) of \(\frac{\mu_5}2\) in (\ref{c_j2}). 
\vskip 25pt
\section{The last step of the proof}\label{the_exp_part}
In this section, we analyze the exponential part of Nakayashi's expression  (\ref{det_expression}), 
and complete the proof of the main result. 
Let
\begin{equation*}
\tilde{f}(t,s)=s-t^2-G,  \ \ \ G=G(t,s).
\end{equation*}
Concerning (\ref{c_j}), it is sufficient to check only the odd power term of \(t\) 
in the expansion of \(\omega_{w_g}(t)\) with respect to \(t\). 
Because of the identity
\begin{equation*}
\frac{\ 1\ }{\ \tilde{f}_s\ }=\frac1{1-G_s}
=(1+G_s)(1+{G_s}^2)(1+{G_s}^{2^2})(1+{G_s}^{2^3})\cdots, \ \ \ \ G_s=\tfrac{\partial}{\partial s}G, 
\end{equation*}
the fact that the polynomial \(G(t,s)\) is a multiple of \(s^2\), 
and the general equality \((X+Y)^{2^k}\equiv X^{2^k}+Y^{2^k}\mod{2}\), letting
\begin{equation*}
\tilde{f}=s-t^2-\sum_{j\geqq 2}f_j(t)s^j, \ \ \ \ 
(\ f_j(t)\in\mathbb{Z}[\bvec{\mu}][t]\ ), 
\end{equation*}
we have
\begin{equation*}
G_s=1-\tilde{f}_s=\sum_{j\geqq 2}j\,f_j(t)s^{j-1}. 
\end{equation*}
\begin{lemma}
Writing down \(\omega_{w_g}(t)\)  in terms of \(t\) and \(s\), it is of type
\begin{equation*}
\frac{t^{\mathrm{\,even}}\,s^{\mathrm{\,odd}}}{\tilde{f}_s}dt \ \ \mbox{or}\ \  
\frac{t^{\mathrm{\,even}}\,s^{\mathrm{\,even}}}{\tilde{f}_s}dt. 
\end{equation*}
On the other hand the pair of parities of powers of \(x\) and \(y\) in a term of \(f(x,y)\)
corresponds to the pair of parities of powers of \(t\) and \(s\) in a term of  \(\tilde{f}(t,s)\). 
More precisely, a term of the type \(x^{\mathrm{odd}}y^{\mathrm{odd}}\)
corresponds the term of the type \(t^{\mathrm{\,odd}}s^{\mathrm{\,even}}\) in the former case, 
and does the terms of the type \(t^{\mathrm{\,odd}}s^{\mathrm{\,odd}}\) in the latter case. 
\end{lemma}
\begin{proof}
Of course, in any case, the parity of \(a\) and \(b\) is different. 
If \(e\) is even, then both \(q\) and \(b\) are odd. 
In this case, \(c\) is odd and \(d\) is even. 
If \(e\) is odd, then \(d\) is odd. 
Summarizing the deduction, in each case separated by the parities of  \(e\), \(a\), \(b\), \(c\), \(d\), 
it is easy to check the claim by using (\ref{f_tilde}) and (\ref{f_tilde_form}). 
\end{proof}
We note here that, in 
\begin{equation*}
\log\sqrt{\frac{\omega_{w_g}(t)}{t^{2g-2}dt}}=\sum_{j=1}^{\infty}\frac{c_j}{j}t^{j-1}, 
\end{equation*}
namely, in the expression
\begin{equation*}
\sum_{j=1}^{\infty}{c_j}t^{j-1}
=\frac12\frac{d}{dt}\log\frac{\omega_{w_g}(t)}{t^{2g-2}dt}=
\frac{\frac12\frac{d}{dt}(\omega_{2g-2}(t)/dt)}
     {\omega_{2g-2}(t)/dt}
=\frac{\frac12\frac{d}{dt}(1+\cdots)}{1+\cdots},
\end{equation*}
all the coefficients in the denominators belongs to \(\mathbb{Z}[\bvec{\mu}]\). 
In the coefficients in the series inside \(\frac{d}{dt}\), 
the coefficients of the term of type \(t^{\mathrm{even}}\) is canceled out by the denominator \(2\) of \(\frac12\).  
Therefore, the critical case occurs only concerning the coefficients of odd power terms 
in the expansion of \(\omega_{w_g}(t)\). 
\par
In the following, we check the two cases in Lemma above separately. 
\vskip 5pt
\noindent
{\it The former case}. In 
\begin{equation*}
\frac{s^{\mathrm{odd}}t^{\mathrm{even}}}{\tilde{f}_s}=t^{\mathrm{even}}s^{\mathrm{even}}\,\frac{s}{1-G_s}
=t^{\mathrm{even}}s^{\mathrm{even}}\,s(1+G_s)(1+{G_s}^2)(1+{G_s}^{2^2})(1+{G_s}^{2^3})\cdots, 
\end{equation*}
it is sufficient to check \(s(1+G_s)\). 
As an element in \(\mathbb{Z}[\bvec{\mu}][[t]]\), we have
\begin{equation*}
\begin{aligned}
s(1+G_s)&=s+sG_s\\
&\equiv t^2+G+sG_s \mod{2}\\
&\equiv t^2+G+\mbox{\lq\lq\,the sum of the terms of type  \(s^{\mathrm{odd}}\) in \(G\)\,"} \mod{2}\\
&\equiv \mbox{\lq\lq\,the sum of the terms of type  \(s^{\mathrm{even}}\)  in \(\tilde{f}\)\,"} \mod{2}.
\end{aligned}
\end{equation*}
In this situation, by replacing the coefficient \(\mu_j\) by \(\mu'_j/2\) in the set \(\bvec{\mu}\)
if and only if the \(\mu_j\) occurs as the coefficient of any term 
of type \(t^{\mathrm{odd}}s^{\mathrm{even}}\) in \(\tilde{f}\)
(this corresponds a term of type \(x^{\mathrm{odd}}y^{\mathrm{odd}}\)), 
we have the replaced set \(\bvec{\mu}'\). 
Then the coefficient of any term of type  \(t^{\mathrm{odd}}\) in 
\(s(1+G_s)\in\mathbb{Z}[\bvec{\mu}'][[t]]\) belongs to \(2\mathbb{Z}[\bvec{\mu}']\). 
As a result, we see that all the coefficients of the expansion of 
\begin{equation*}
\frac12\frac{d}{dt}\log\bigg(\frac{\omega_{w_g}(t)}{t^{2g-2}dt}\bigg)
\end{equation*}
with respect to \(t\) belong to  \(\mathbb{Z}[\bvec{\mu}']\). 
Hence, \(\forall{c_j}\in\mathbb{Z}[\bvec{\mu}']\). 
Now our claim is proved for the first case. 
\vskip 10pt
\noindent
{\it The latter case}. Since 
\begin{equation*}
\frac{s^{\mathrm{even}}t^{\mathrm{even}}}{\tilde{f}_s}=t^{\mathrm{even}}s^{\mathrm{even}}\,\frac{1}{1-G_s}
=t^{\mathrm{even}}s^{\mathrm{even}}\,(1+G_s)(1+{G_s}^2)(1+{G_s}^{2^2})(1+{G_s}^{2^3})\cdots.
\end{equation*}
It is sufficient to check \,\((1+G_s)\). 
The crucial case may occur for the coefficient
of any term of the type  \(t^{\mathrm{odd}}s^{\mathrm{odd}}\) in the expansion of \(\tilde{f}\). 
Our claim is proved by the same argument in the former case, 
and we have completed the proof. 
\vskip 20pt
\ \ 

\bibliographystyle{amsplain}
\addcontentsline{toc}{section}{Bibliography}
\bibliography{refers}

\providecommand{\bysame}{\leavevmode\hbox to3em{\hrulefill}\thinspace}
\providecommand{\MR}{\relax\ifhmode\unskip\space\fi MR }
\providecommand{\MRhref}[2]{%
  \href{http://www.ams.org/mathscinet-getitem?mr=#1}{#2}
}
\providecommand{\href}[2]{#2}
\begin{thebibliography}{10}

\bibitem{baker_1897}
{Baker,~H.F.}, \emph{{Abelian functions -- Abel's theorem and the allied theory
  including the theory of the theta functions --}}, Cambridge Univ. Press,
  1897.

\bibitem{bannai-kobayashi_2006}
{Bannai,~K. and Kobayashi,~S.}, \emph{{Divisibilities of Eisenstein-Kronecker
  numbers and $p$-adic theta functions at supersingular primes}}, unpubilshed
  manuscript.

\bibitem{barsotti_1970}
{Barsotti,~I.}, \emph{{Considerazioni sulle funzioni th\^eta}}, {Istituo
  Nazionale di alta mathematica, Symposia Mathematica} \textbf{III} (1970),
  247--277.

\bibitem{breen}
{Breen, L.}, \emph{{Fonctions th\^eta et th\'eor{\accent 18 e}me du cube}},
  Lecture Notes in Mathematics, vol. 980, Springer-Verlag, 1983.

\bibitem{bl_2004}
{Buchstaber,~V.M. and Leykin,~D.V.}, \emph{{Heat equations in a nonholonomic
  frame}}, {Functional Anal. Appl.,} \textbf{38} (2004), 88--101.

\bibitem{eemop_2009}
{Eilbeck,~J.C., Enol'skii,~V.Z., Matsutani,~S., \^Onishi,~Y. and Previato,~E.},
  \emph{{Abelian functions for trigonal curves of genus three}}, {International
  Mathematics Research Notices,} \textbf{{2008}} ({2008}), {102--139}.

\bibitem{fay_1973}
{Fay,~J.}, \emph{{Theta functions on Riemann surfaces}}, Lect. Notes in Math.,
  vol. 352, Springer-Verlag, 1973.

\bibitem{macdonald_1995}
{Macdonald,~I.G.}, \emph{{Symmetric functions and Hall polynomials}},
  {Clarendon Press}, Oxford, 1995.

\bibitem{mazur-tate_1991}
{Mazur,~B. and Tate,~J.}, \emph{{The $p$-adic sigma function}}, Duke Math. J.,
  \textbf{62} (1991), 663--688.

\bibitem{MST_2006}
{Mazur,~B., Stein,~W. and Tate,~J.}, \emph{Computation of $p$-adic height and
  log convergence}, Documenta math., \textbf{{\rm Extra Volume : John H.
  Coates' Sixtieth Birthday}} (2006), 577--614.

\bibitem{miura_1998}
{Miura, S.}, \emph{{Linear code on affine algebraic curves {\rm (in Japanese,
  Affine Dai-su-kyoku-sen-zyo no sen-kei-fu-gou)}}}, {Journal of the institute
  of electorics, information and communication engineering A,} \textbf{{\rm
  J81-A}} (1998), 1398--1421.

\bibitem{nakayashiki_2010}
{Nakayashiki,~A.}, \emph{{Sigma function as a tau function}}, {International
  Mathematical Research Notices,} \textbf{2010} (2010), 373--394.

\bibitem{nakayashiki_2008}
{Nayakashiki, A.}, \emph{{On algebraic expressions of sigma functions for
  \((n,s)\) curves}}, {Asian J. Math.} \textbf{14} (2010), no.~2, 175--212.

\bibitem{lect_chuo}
{\^Onishi,~Y.}, \emph{{\rm Theory of Abelian functions (in Japanese, Abel
  Kan-su-ron)}}, Department of Math., Chuo University, 2013.

\end{thebibliography}
\end{document}